\documentclass[reqno,a4paper]{amsart}
\usepackage{amssymb}
\usepackage{amsmath}
\begin{document} 
\newtheorem{prop}{Proposition}[section]
\newtheorem{Def}{Definition}[section] \newtheorem{theorem}{Theorem}[section]
\newtheorem{lemma}{Lemma}[section] \newtheorem{Cor}{Corollary}[section]

\title[Nonlinear Dirac equation]{\bf Local well-posedness for the nonlinear Dirac equation in two space dimensions}
\author[Hartmut Pecher]{
{\bf Hartmut Pecher}\\
Fachbereich Mathematik und Naturwissenschaften\\
Bergische Universit\"at Wuppertal\\
Gau{\ss}str.  20\\
42097 Wuppertal\\
Germany\\
e-mail {\tt pecher@math.uni-wuppertal.de}}
\date{}

\begin{abstract}
The Cauchy problem for the cubic nonlinear Dirac equation in two space dimensions is locally well-posed for data in $H^s$ for $ s > 1/2$. The proof given in spaces of Bourgain-Klainerman-Machedon type relies on the null structure of the nonlinearity as used by d'Ancona-Foschi-Selberg for the Dirac-Klein-Gordon system before and bilinear Strichartz type estimates for the wave equation by Selberg and Foschi-Klainerman.
\end{abstract}
\maketitle
\renewcommand{\thefootnote}{\fnsymbol{footnote}}
\footnotetext{\hspace{-1.5em}{\it 2000 Mathematics Subject Classification:} 
35Q55, 35L70 }
\normalsize 
\setcounter{section}{0}
\section{Introduction and main results}
Consider the Cauchy problem for the nonlinear Dirac
equation in two space dimensions 
\begin{equation}
\label{1}
i(\partial_t + \alpha \cdot \nabla) \psi + M \beta \psi  =  - \langle \beta \psi,\psi \rangle \beta \psi 
\end{equation}
with initial data
\begin{equation}
\psi(0)  =  \psi_0 \, .
\label{2}
\end{equation}
Here $\psi$ is a two-spinor field, i.e. $\psi : {\mathbb R}^{1+2} \to {\mathbb C}^2$,  
$M 
\in {\mathbb R}$ and $\nabla = (\partial_{x_1} , \partial_{x_2}) $ , $ \alpha \cdot 
\nabla = \alpha^1 \partial_{x_1} + \alpha^2 \partial_{x_2}$ . 
$\alpha^1,\alpha^2, \beta$ are hermitian ($ 2 \times 2$)-matrices satisfying 
$\beta^2 = 
(\alpha^1)^2 = (\alpha^2)^2 = I $ , $ \alpha^j \beta + \beta \alpha^j = 0, $  $ 
\alpha^j \alpha^k + \alpha^k \alpha^j = 2 \delta^{jk} I $ . \\
$\langle \cdot,\cdot \rangle $ denotes the ${\mathbb C}^2$ - scalar product. A 
particular representation is given by \\
$ \alpha^1 = \left( \begin{array}{cc}
0 & 1  \\ 
1 & 0  \end{array} \right)$ 
 , $ \alpha^2 = \left( \begin{array}{cc}
0 & -i  \\
i & 0  \end{array} \right)$ , $ \beta = \left( \begin{array}{cc}
1 & 0  \\
0 & -1  \end{array} \right)$ .

We consider Cauchy data in Sobolev spaces: $\psi_0 \in H^s({\mathbb R}^2) $ .

In quantum field theory the nonlinear Dirac equation is a model of self-interacting Dirac fermions. It was originally formulated in one space dimension known as the Thirring model \cite{T} and in three space dimensions \cite{So}. See also \cite{FLR}, \cite{FFK}, \cite{GN}.

In the case of one space dimension global existence for data in $H^1$ was proven by Delgado \cite{D}. For less regular data Selberg and Tesfahun \cite{ST} showed local wellposedness in $H^s$ for $s>0$, unconditional uniqueness in $C^0([0,T],H^s)$ for $s > 1/4$ and global well-posedness for $s > 1/2$. Recently T.Candy \cite{C} was able to show global well-posedness in $L^2$, which is the critical case with respect to scaling.

In the case of three space dimensions Escobedo and Vega \cite{EV} showed local well-posedness in $H^s$ for $s>1$, which is almost critical with respect to scaling. Moreover they considered more general nonlinearities, too. Global solutions for small data in $H^s$ for $s>1$ were shown to exist by Machihara, Nakanishi and Ozawa \cite{MNO}. Machihara, Nakamura, Nakanishi and Ozawa \cite{MNNO} proved global existence for small data in $H^1$ under some additional regularity assumptions for the angular variables.

In the present paper we now consider the case of two space dimensions where the critical space is $H^{1/2}$. We show local well-posedness in $H^s$ for $ s > 1/2$, which is optimal up to the endpoint, and unconditional uniqueness for $s > 3/4$. We construct the solutions in spaces of Bougain-Klainerman-Machedon type, using that the nonlinearity satisfies a null condition. Our proof uses the approach to the corresponding problem for the Dirac-Klein-Gordon equations by d'Ancona, Foschi and Selberg \cite{AFS},\cite{AFS1}. The crucial estimates for the cubic nonlinearity can then be reduced to bilinear Strichartz type estimates for the wave equation which were given by S. Selberg \cite{S} and D. Foschi and S. Klainerman \cite{FK}.

It is possible to simplify the system (\ref{1}),(\ref{2}) by 
considering the projections onto the one-dimensional eigenspaces of the 
operator 
$-i \alpha \cdot \nabla$ belonging to the eigenvalues $ \pm |\xi|$. These 
projections are given by $\Pi_{\pm}(D)$, where  $ D = 
\frac{\nabla}{i} $ and $\Pi_{\pm}(\xi) = \frac{1}{2}(I 
\pm \frac{\xi}{|\xi|} \cdot \alpha) $. Then $ 
-i\alpha \cdot \nabla = |D| \Pi_+(D) - |D| \Pi_-(D) $ and $ \Pi_{\pm}(\xi) \beta
= \beta \Pi_{\mp}(\xi) $. Defining $ \psi_{\pm} := \Pi_{\pm}(D) \psi$  , the Dirac  equation can be rewritten as
\begin{equation}
\label{4}
(-i \partial_t \pm |D|)\psi_{\pm}  =  -M\beta \psi_{\mp} + \Pi_{\pm}( \langle \beta (\psi_+ + \psi_-), \psi_+ + \psi_- \rangle
\beta (\psi_+ + \psi_-)) 
\end{equation}
The initial condition is transformed into
\begin{equation}
\label{6}
\psi_{\pm}(0) = \Pi_{\pm}(D)\psi_0 \, .
\end{equation}
We consider the integral equations belonging to the 
Cauchy problem (\ref{4}),(\ref{6}):
\begin{align}
\nonumber
\psi_{\pm}(t) & =  e^{\mp it|D|} \psi_{\pm}(0) 
-i \int_0^t e^{\mp i(t-s)|D|} 
\Pi_{\pm}(D)(\langle \beta(\Pi_+(D) \psi_+(s)+\Pi_-(D)\psi_-(s), \\ \label{7}
 &\Pi_+(D)\psi_+(s) + \Pi_-(D)\psi_-(s) \rangle \beta(\Pi_+(D) \psi_+(s) 
 + \Pi_-(D)\psi_-(s)) ds \\ \nonumber
 &+ iM \int_0^t e^{\mp i(t-s)|D|} \beta 
\psi{\mp}(s) ds \, .
\end{align}
We remark that any solution of this system automatically fulfills 
$\Pi_{\pm}(D)\psi_{\pm} = \psi_{\pm}$, because applying $\Pi_{\pm}(D)$ to the 
right hand side of (\ref{7}) gives $\Pi_{\pm}(D)\psi_{\pm}(0) = 
\psi_{\pm}(0)$ 
and the integral terms also remain unchanged, because $\Pi_{\pm}(D)^2 = 
\Pi_{\pm}(D)$ and $\Pi_{\pm}(D) \beta \psi_{\mp}(s) = \beta 
\Pi_{\mp}(D)\psi_{\mp}(s) = \beta\psi_{\mp}(s)$. Thus $\Pi_{\pm}(D)\psi_{\pm}$ 
can be 
replaced by $\psi_{\pm}$, thus the system of integral equations reduces exactly 
to the one belonging to our Cauchy problem (\ref{4}),(\ref{6}).

We use the following function spaces and notation. Let $\, \widehat{} \,$ denote the Fourier 
transform with respect to space  and $\, \tilde{} \,$ and $ \, \check{} \,$ the Fourier 
transform and its inverse, respectively, with respect to space and time simultaneously.
The standard spaces of Bougain-Klainerman-Machedon type belonging 
to 
the half waves are defined by the completion of ${\mathcal S}({\mathbb R} \times {\mathbb 
R^2})$ with respect to
$$ \|f\|_{X^{s,b}_{\pm}} = \|U_{\pm}(-t)f\|_{H^b_t H^s_x} = \| \langle \xi 
\rangle^s \langle \tau \pm |\xi| \rangle^b \tilde{f}(\tau,\xi) \|_{L^2} \, $$
where 
$$ U_{\pm}(t):=e^{\mp it|D|} \quad {\mbox {and}} \quad \|g\|_{H^b_t H^s_x} = \| 
\langle 
\xi \rangle^s \langle \tau \rangle^b \tilde{g}(\xi,\tau)\|_{L^2_{\xi,\tau}} \, 
. 
$$
We also define $X^{s,b}_{\pm}[0,T]$ as the space of restrictions of functions in $X^{s,b}_{\pm}$ to the time interval $[0,T]$ with norm
$ \|f\|_{X^{s,b}_{\pm}[0,T]} = \inf_{\tilde{f}_{|[0,T]} =f} \| \tilde{f}\|_{X^{s,b}_{\pm}} \, . $

We use the Strichartz estimates for the homogeneous wave equation in ${\mathbb R^n} 
\times {\mathbb R}$, which can be found e.g. in Ginibre-Velo \cite{GV}, Prop. 2.1.
\begin{prop}
\label{Prop.1.1}
Let $\gamma(r) = (n-1)(\frac{1}{2}-\frac{1}{r})$ , $ \delta(r) = 
n(\frac{1}{2}-\frac{1}{r}) $ , $ n \ge 2 $. Let $\rho,\mu \in {\mathbb R}$, $2 \le 
q \le \infty $, $2 \le r < \infty $ satisfy $0 \le \frac{2}{q} \le \min(\gamma(r),1) $ , $ 
(\frac{2}{q},\gamma(r)) \neq (1,1)$, $ \rho + \delta(r) - \frac{1}{q} = \mu $. 
Then
$$ \|e^{\pm it|D|}u_0\|_{L^q({\mathbb R}, \dot{H}^{\rho}_{r}({\mathbb R}^n))} \le c 
\|u_0\|_{\dot{H}^{\mu}({\mathbb R}^n)} \, . $$
\end{prop}

Fundamental for our results are the following bilinear Strichartz type estimates, which we state for the two-dimensional case.
\begin{prop}
\label{Prop.1.2}
With the notation $(|D|^{\alpha} f)\widehat{}(\xi) = |\xi|^{\alpha} \widehat{f}(\xi)$ and $(D_-^{\alpha} F)\tilde{}(\tau,\xi) = ||\tau|-|\xi||^{\alpha} \tilde{F}(\tau,\xi)$ the following estimate holds for independent signs $\pm$ and $\pm_1$:
$$ \| |D|^{\beta_0} D_-^{\beta_-}(e^{\pm it|D|} u_0 \,e^{\pm_1 it|D|} v_0)\|_{L^2(\mathbb{R} \times \mathbb{R}^2)} \lesssim \|u_0\|_{\dot{H}^{\alpha_1}(\mathbb{R}^2)} \|v_0\|_{\dot{H}^{\alpha_2}(\mathbb{R}^2)} $$
if and only if the following conditions are satisfied:\\
$ \beta_0 + \beta_- = \alpha_1 + \alpha_2 - \frac{1}{2}$ , $\beta_- \ge \frac{1}{4}$ , $ \beta_0 > -\frac{1}{2}$ , $\alpha_i \le \beta_- + \frac{1}{2}   \, (i=1,2)$ , $ \alpha_1 + \alpha_2 \ge \frac{1}{2},$  $(\alpha_i,\beta_-) \neq (\frac{3}{4},\frac{1}{4})  \,  (i=1,2)$ , $ (\alpha_1 + \alpha_2,\beta_-) \neq(\frac{1}{2},\frac{1}{4})$ .
\end{prop}
\begin{proof}
\cite{FK}, Theorem 1.1
\end{proof}
The so-called transfer principle immediately implies
\begin{Cor}
\label{Cor.1.1}
Under the assumptions of the proposition the following estimate holds:
$$ \| |D|^{\beta_0} D_-^{\beta_-} (fg)\|_{L^2({\mathbb R}\times {\mathbb R}^2)} \lesssim \| |D|^{\alpha_1} f\|_{X_{\pm}^{0,\frac{1}{2}+}}\| |D|^{\alpha_2} g\|_{X_{\pm_1}^{0,\frac{1}{2}+}} \, . $$
\end{Cor}
We also need the following improvement for products of the type (+,+) and (--,--):
\begin{prop}
\label{Prop.1.3}
The following estimate holds for equal signs:
$$ \| |D|^{\beta_0} D_-^{\frac{1}{4}}(e^{\pm it|D|} u_0 \, e^{\pm it|D|} v_0\|_{L^2(\mathbb{R} \times \mathbb{R}^2)} \lesssim \|u_0\|_{\dot{H}^{\alpha_1}(\mathbb{R}^2)} \|v_0\|_{\dot{H}^{\alpha_2}(\mathbb{R}^2)} \, ,$$
under the assumptions
$ \beta_0 = \alpha_1 + \alpha_2 - \frac{3}{4}$ , $ \alpha_1,\alpha_2 < \frac{3}{4}$ , $ \alpha_1 + \alpha_2 > \frac{1}{4}$ .
\end{prop}
\begin{proof}
\cite{S}, Theorem 6(b) or \cite{FK}, Theorem 12.1 (see also \cite{AFS1}, formula (15)).
\end{proof}
\begin{Cor}
\label{Cor.1.2}
Under the assumptions of the proposition the following estimate holds:
$$ \| |D|^{\beta_0} D_-^{\frac{1}{4}} (fg)\|_{L^2({\mathbb R}\times {\mathbb R}^2)} \lesssim \| |D|^{\alpha_1} f\|_{X_{\pm}^{0,\frac{1}{2}+}}\| |D|^{\alpha_2} g\|_{X_{\pm}^{0,\frac{1}{2}+}} \, . $$
\end{Cor}
The main result reads as follows:
\begin{theorem}
\label{Theorem1.1}
The Cauchy problem for the Dirac equation (\ref{1}), 
(\ref{2}) has a unique local solution $\psi$ for data $\psi_0 \in H^s({\mathbb R}^2)$, if $ s > 1/2$. 
More precisely there exists a $T >0$ and a unique solution
$$ \psi \in X^{s,\frac{1}{2}+}_+ [0,T] + X^{s,\frac{1}{2}+}_- [0,T]\, . $$
This solution has the property
$$ \psi \in C^0([0,T],H^s({\mathbb R}^2)) \, . $$
\end{theorem}
We also get the following uniqueness result.
\begin{theorem} \label{Theorem1.2}
The solution of Theorem \ref{Theorem1.1} is (unconditionally) unique in the space $C^0([0,T],H^s({\mathbb R}^2))$, if $ s > 3/4 $.
\end{theorem}
We use the following well-known linear estimates (cf. e.g. \cite{AFS}, Lemma 5).
\begin{prop}
\label{Prop.1.4}
Let $ 1/2 < b \le 1$ , $ s \in {\mathbb R}$ , $0<T\le 1$ and $0 \le \delta \le 1-b$. The Cauchy problem
$$ (-i\partial_t \pm |D|) \psi_{\pm} = F  \quad , \quad \psi_{\pm}(0) = f $$
for data $F \in X^{s,b-1+\delta}_{\pm}[0,T]$ and $f \in H^s$ has a unique solution $\psi_{\pm} \in X^{s,b}_{\pm}[0,T]$. It fulfills
$$ \|\psi_{\pm}\|_{X^{s,b}_{\pm}[0,T]} \lesssim \|f\|_{H^s} + T^{\delta} \|F\|_{X^{s,b-1+\delta}_{\pm}[0,T]} $$
with an implicit constant independent of $T$.
\end{prop}

Finally we use the following notation: $\langle \cdot \rangle = 1 + |\cdot|$. For $a \in {\mathbb R}$ and $\epsilon > 0$ we denote by $a+,a++,a-,a--$ numbers with $a-\epsilon < a-- < a- < a < a+ < a++ <a+\epsilon$.
  
\section{Proof of the Theorems}
\begin{proof}[{\bf Proof of Theorem \ref{Theorem1.1}}]
Using Prop. \ref{Prop.1.4} a standard application of the contraction mapping principle reduces the proof to the estimates for the nonlinearity in the following Proposition \ref{Prop.2.1}.
\end{proof}

\begin{prop}
\label{Prop.2.1}
For any $\epsilon > 0$ the following estimate holds:
$$ \|\Pi_{\pm}(\langle \beta \Pi_{\pm_1} \psi_1,\Pi_{\pm_2} \psi_2 \rangle \beta \Pi_{\pm_3} \psi_3)\|_{X_{\pm}^{\frac{1}{2}+\epsilon,-\frac{1}{2}++}} \lesssim \prod_{i=1}^3 \|\psi_i\|_{X^{\frac{1}{2}+\epsilon,\frac{1}{2}+}_{\pm_i}} \, . $$
Here and in the following $\pm,\pm_1,\pm_2,\pm_3$ denote independent signs.
\end{prop}

The null structure of the Dirac equation has the following consequences (we here follow closely \cite{AFS} and \cite{AFS1}). 
Denoting
$$ \sigma_{\pm ,\pm 3}(\eta,\zeta) := \Pi_{\pm_3}(\zeta) \beta \Pi_{\pm 
}(\eta) = \beta \Pi_{\mp_3}(\zeta) \Pi_{\pm  }(\eta) \, , $$
we remark that by orthogonality this quantity vanishes if $\pm \eta$ and 
$\pm_3  \zeta$ line up in the same direction whereas in general (cf. 
\cite{AFS1}, Lemma 1):
\begin{lemma}
\label{Lemma2.1}
$$\sigma_{\pm ,\pm 3}(\eta,\zeta) = O(\angle (\pm \eta, \pm_3 \zeta)) \, ,$$ 
where $\angle(\eta,\zeta)$ denotes the angle between the vectors $\eta$ and 
$\zeta$.
\end{lemma}
Consequently we get
\begin{align} \nonumber
\lefteqn{ |\langle \beta \Pi_{\pm_3 }(D)\psi_3 ,\Pi_{\pm}(D)\psi_0  
  \rangle ^{\tilde{}}(\tau,\xi)|} \\ \nonumber
& \le  \int |\langle \beta \Pi_{\pm_3}(\eta)\tilde \psi_3(\lambda,\eta), 
\Pi_{\pm}(\eta - \xi) \tilde{\psi}_0(\lambda - \tau,\eta - \xi)\rangle| d\lambda d\eta 
\\ \label{40}
& =  
\int |\langle \Pi_{\pm}(\eta - \xi) \beta \Pi_{\pm_3}(\eta)\tilde 
{\psi}_3(\lambda,\eta),  \tilde{\psi}_0(\lambda - \tau,\eta - \xi)\rangle| d\lambda 
d\eta\\ \nonumber
& \lesssim   \int \Theta_{\pm,\pm 3} \, |\tilde{\psi}_1(\lambda,\eta)|\, 
|\tilde{\psi}_2(\lambda - \tau, \eta - \xi)| d\lambda d\eta \, ,
\end{align} 
where we denote
$\Theta_{\pm,\pm_3} = \angle(\pm_3 \eta,\pm(\eta - \xi))$ and $\overline{\Theta}_{\pm_1,\pm_2} = \angle(\pm_2\zeta,\pm_1(\zeta - \xi))$.

We also need the following elementary estimates which can be found in 
\cite{AFS}, section 5.1 or \cite{GP}, Lemma 3.2 and Lemma 3.3.
\begin{lemma}
\label{Lemma2.2}
Denoting
$$ A = |\tau| - |\xi| \, , \, B_{\pm} = \lambda \pm |\eta| \, , \, C_{\pm} = 
\lambda - \tau \pm |\eta - \xi| \, , \, \Theta_{\pm} = 
\angle(\eta,\pm(\eta-\xi)) $$
and
$$ \rho_+ = |\xi| - \left| |\eta| - |\eta-\xi| \right| \, , \, \rho_- = 
|\eta|+|\eta-\xi| - |\xi | $$
the following estimates hold:
$$ \Theta_+^2 \sim \frac{|\xi| \rho_+}{|\eta| |\eta - \xi|} \, , \, \Theta_-^2 
\sim \frac{(|\eta|+|\eta-\xi|)\rho_-}{|\eta| |\eta-\xi|} \sim 
\frac{\rho_-}{\min(|\eta|,|\eta - \xi|)} $$ 
as well as
$$ \rho_{\pm} \le 2 \min(|\eta|,|\eta-\xi|) $$
and 
$$ \rho_ {\pm} \le |A| + |B_+| + |C_{\pm}| $$
as well as
$$ \rho_ {\pm} \le |A| + |B_-| + |C_{\mp}| \, .$$
Similarly we define
$$ D_{\pm} = \sigma \pm |\zeta| \, , \, E_{\pm} = \sigma - \tau \pm |\zeta - \xi| \, , \, \overline{\Theta}_{\pm} = \langle (\zeta,\pm(\zeta - \xi)) $$
and
$$ \overline{\rho}_+ = |\xi| - \left| |\zeta| - |\zeta-\xi| \right| \, , \, \overline{\rho}_- = 
|\zeta|+|\zeta-\xi| - |\xi | $$
then the following estimates hold:
$$ \overline{\Theta}_+^2 \sim \frac{|\xi| \overline{\rho}_+}{|\zeta| |\zeta - \xi|} \, , \, \overline{\Theta}_-^2 
\sim \frac{(|\zeta|+|\zeta-\xi|)\overline{\rho}_-}{|\zeta| |\zeta-\xi|} \sim 
\frac{\overline{\rho}_-}{\min(|\eta|,|\eta - \xi|)} $$ 
as well as
$$ \overline{\rho}_{\pm} \le 2 \min(|\zeta|,|\zeta-\xi|) $$
and 
$$\overline{\rho}_ {\pm} \le |A| + |D_+| + |E_{\pm}| $$
as well as
$$ \overline{\rho}_ {\pm} \le |A| + |D_-| + |E_{\mp}| \, .$$
\end{lemma} 

\begin{proof}[{\bf Proof of Prop. \ref{Prop.2.1}}]
The claim of the proposition is equivalent to the estimate
$$ |\int \langle \langle \beta \Pi_{\pm_1} \psi_1,\Pi_{\pm_2} \psi_2 \rangle \beta \Pi_{\pm_3} \psi_3, \Pi_{\pm} \psi_0 \rangle dxdt | \lesssim \prod_{i=1}^3 \|\psi_i\|_{X^{\frac{1}{2}+\epsilon,\frac{1}{2}+}_{\pm_i}} \|\psi_0\|_{X_{\pm}^{-\frac{1}{2}-\epsilon,\frac{1}{2}--}} \, . $$
The left hand side equals
$$ \left|\int \langle \beta \Pi_{\pm_1} \psi_1,\Pi_{\pm_2}\psi_2 \rangle\tilde{}\, \langle \beta \Pi_{\pm_3} \psi_3, \Pi{\pm} \psi_0 \rangle\tilde{}\, d\tau d\xi \right|\, .$$
Using (\ref{40}) it is thus sufficient to prove
\begin{align*}
 \left| \int  \Theta_{\pm,\pm_3} \tilde{\psi}_3(\lambda,\eta) 
\tilde{\psi}_0(\lambda-\tau,\eta-\xi)\overline{\Theta}_{\pm_1,\pm_2} \tilde{\psi}_1(\sigma,\zeta) 
\tilde{\psi}_2(\sigma-\tau,\zeta-\xi) d\sigma d\zeta  d\tau d\xi d\eta d\lambda \right| \\
\lesssim \prod_{i=1}^3 \|\psi_i\|_{X^{\frac{1}{2}+\epsilon,\frac{1}{2}+}_{\pm_i}} \|\psi_0\|_{X_{\pm}^{-\frac{1}{2}-\epsilon,\frac{1}{2}--}} \, , 
\end{align*} 
where we assume w.l.o.g. that the Fourier transforms are nonnegative.
 Defining
\begin{align*}
\tilde{F}_j(\lambda,\eta) &:= \langle \eta \rangle^{\frac{1}{2}+\epsilon} \langle \lambda \pm_j |\eta| \rangle^{\frac{1}{2}+} 
\tilde{\psi}_j(\lambda,\eta) \quad (j=1,2,3) \\
\tilde{F_0}(\lambda,\eta) &:= \langle \eta \rangle^{-\frac{1}{2}-\epsilon} \langle \lambda \pm |\eta| 
\rangle^{\frac{1}{2}--} \tilde{\psi}_0(\lambda,\eta) 
\end{align*}
we thus have to show
\begin{align} \nonumber
I & :=  \int  \Theta_{\pm,\pm_3} 
\frac{\tilde{F}_3(\lambda,\eta)}{\langle \eta \rangle^{\frac{1}{2}+\epsilon}
 \langle B_{\pm_3} \rangle ^{\frac{1}{2}+}}
 \frac{\tilde{F}_0(\lambda-\tau,\eta-\xi) \langle \eta - \xi \rangle^{\frac{1}{2}+\epsilon}}{\langle C_{\pm} \rangle 
^{\frac{1}{2}--}} \\ \nonumber
&  \hspace{2cm}\overline{\Theta}_{\pm_1,\pm_2} \frac{\tilde{F}_1(\sigma,\zeta)}{\langle \zeta \rangle^{\frac{1}{2}+\epsilon} \langle D_{\pm_1} \rangle ^{\frac{1}{2}+}}
\frac{\tilde{F}_2(\sigma-\tau,\zeta-\xi)}{\langle \zeta - \xi \rangle^{\frac{1}{2}+\epsilon}\langle E_{\pm_2} \rangle 
^{\frac{1}{2}+}}d\sigma d\zeta d\tau d\xi d\eta d\lambda \\ \label{41}
& \lesssim \prod_{i=0}^3 \|F_i\|_{L^2_{xt}} \, . 
\end{align}  
In order to prove (\ref{41}) let us first of all consider the low frequency case $|\eta - \xi| \le 1$. We simply use $|\Theta_{\pm,\pm_3}|,|\overline{\Theta}_{\pm_1,\pm_2}| \lesssim 1$ and estimate crudely
\begin{align*}
I & \le \| (\frac{\tilde{F}_3(\lambda,\eta)}{\langle \eta \rangle^{\frac{1}{2}+\epsilon} \langle B_{\pm_3}\rangle^{\frac{1}{2}+}})\check{\,\,}\|_{L^6_{xt}}
\|F_0\|_{L^2_{xt}} \| (\frac{\tilde{F}_1(\sigma,\zeta )}{\langle \zeta \rangle^{\frac{1}{2}+\epsilon} \langle D_{\pm_1}\rangle^{\frac{1}{2}+}})\check{\,\,}\|_{L^6_{xt}} \\ &\hspace{2em} \|(\frac{\tilde{F}_2}{\langle \zeta -\xi \rangle^{\frac{1}{2}+\epsilon} \langle E_{\pm_2}\rangle^{\frac{1}{2}+}})\check{\,\,}\|_{L^6_{xt}} \\
& \lesssim \prod_{i=0}^3 \|F_i\|_{L^2_{xt}}
\end{align*}
using Strichartz' estimate $\,\| e^{\pm it|D|} u_0\|_{L^6_{xt}} \lesssim \|u_0\|_{\dot{H}^{\frac{1}{2}}}$ (Prop. \ref{Prop.1.1}).

From now on we assume $|\eta - \xi| \ge 1$. The estimates for $I$ depend on the different signs which have to be considered.\\
{\bf Part I:} We start with the case where all the signs $\pm,\pm_3,\pm_1,\pm_2$ are + -signs. Analogously one can treat all the cases where $\pm$ and $\pm_3$ as well as $\pm_1$ and $\pm_2$ have the same sign.
Besides the trivial bounds $\Theta_{+,+} , \Theta_{+,+} \lesssim 1$  we make in the following repeated use of the following estimates which immediately follow from Lemma \ref{Lemma2.2}:
\begin{align}
\label{45}
\Theta_{+,+} & \lesssim \frac{|\xi|^{\frac{1}{2}}}{|\eta|^{\frac{1}{2}} |\eta - \xi|^{\frac{1}{2}}} (|A|^{\frac{1}{2}} + |B_+|^{\frac{1}{2}} + |C_+|^{\frac{1}{2}} ) \, ,\\ \label{45'}
\overline{\Theta}_{+,+} & \lesssim \frac{|\xi|^{\frac{1}{2}}}{|\zeta|^{\frac{1}{2}} |\zeta - \xi|^{\frac{1}{2}}} (|A|^{\frac{1}{2}} + |D_+|^{\frac{1}{2}} + |E_+|^{\frac{1}{2}}) \, .
\end{align}
In the case $|C_+| \ge |A|,|B_+|$ we also use
\begin{equation}
\label{46}
\Theta_{+,+}  \lesssim \frac{|\xi|^{\frac{1}{2}}}{|\eta|^{\frac{1}{2}} |\eta - \xi|^{\frac{1}{2}}}  |C_+|^{\frac{1}{2}-} \min(|\eta|,|\eta-\xi|)^{0+} \, .
\end{equation}
We consider several cases depending on the relative size of the terms in the right hand sides of (\ref{45}) and (\ref{45'}). We may assume by symmetry in (\ref{41}) that for the rest of the proof we have $|D_{\pm_1}| \ge |E_{\pm_2}|$, which reduces the number of cases. \\
{\bf Case 1:} $|B_+| \ge |A|,|C_+|$ and $|D_+| \ge |A|,|E_+|$. \\
{\bf Case 1.1:} $\langle C_+ \rangle \le |\xi|$.\\
{\bf Case 1.1.1:} $|\xi| \ll |\eta| $ $\Rightarrow$ $|\eta - \xi| \sim |\eta|$.\\
Using $\Theta_{+,+} \lesssim  \Theta_{+,+}^{1-}$ we obtain
\begin{align*}
I & \lesssim \int \frac{\tilde{F}_3(\lambda,\eta)}{\langle \eta \rangle^{\frac{1}{2}+\epsilon}} 
\frac{\tilde{F}_0(\lambda-\tau,\eta - \xi) \langle \eta - \xi \rangle^{\frac{1}{2}+\epsilon}}{\langle C_+ \rangle^{\frac{1}{2}+}} 
\frac{|\xi|^{0+}|\xi|^{\frac{1}{2}-}}{|\eta|^{\frac{1}{2}-} |\eta - \xi|^{\frac{1}{2}-}} \\
& \hspace{2em} \frac{\tilde{F}_1(\sigma,\zeta)}{\langle \zeta\rangle^{\frac{1}{2}+\epsilon}} 
\frac{\tilde{F}_2(\sigma - \tau, \zeta - \xi)}{\langle E_+ \rangle^{\frac{1}{2}+} \langle \zeta - \xi \rangle^{\frac{1}{2}+\epsilon}}
 \frac{|\xi|^{\frac{1}{2}}}{|\zeta|^{\frac{1}{2}} |\zeta - \xi|^{\frac{1}{2}}}
  d\sigma d\zeta d\tau d\xi d\eta d\lambda \, .
\end{align*}
To reduce the number of cases we always assume $|\zeta| \ge |\zeta - \xi|$, because the alternative case can be treated similarly. Thus we have $|\xi| \lesssim |\zeta|$. We obtain the estimate
\begin{align*}
I  &\lesssim \| ( \frac{\tilde{F}_3(\lambda,\eta)}{|\eta|^{1-}})\check{\,\,}\|_{L^2_t L^{\infty -}_x} 
\|( \frac{\tilde{F}_0(\lambda - \tau,\eta - \xi)}{\langle C_+\rangle^{\frac{1}{2}+}})\check{\,\,}\|_{L^{\infty}_t L^2_x}\| 
( \frac{\tilde{F}_1(\sigma,\zeta )}{\langle \zeta \rangle^{\epsilon -}})\check{\,\,}\|_{L^2_t L^{2+}_x} \\ & \hspace{2em}
\| 
 ( \frac{\tilde{F}_2(\sigma - \tau,\zeta -\xi)}{\langle E_+ \rangle^{\frac{1}{2}+} \langle \zeta - \xi \rangle^{\frac{1}{2}+\epsilon} |\zeta - \xi|^{\frac{1}{2}}})\check{\,\,}\|_{L^{\infty}_t L^{\infty }_x} \\
 & \lesssim \prod_{i=0}^3 \|F_i\|_{L^2_{xt}}
\end{align*}
{\bf Case 1.1.2:} $|\xi| \gtrsim |\eta|$ ($\Rightarrow$ $|\eta - \xi| \lesssim |\xi|$).\\
Estimating $\Theta_{+,+} \lesssim \Theta_{+,+}^{1-2\epsilon}$ and using $|\xi| \lesssim |\zeta|$ we obtain
\begin{align*}
I & \lesssim \int \frac{\tilde{F}_3(\lambda,\eta)}{\langle \eta \rangle^{\frac{1}{2}+\epsilon} \langle B_+ \rangle^{\epsilon}} 
\frac{\tilde{F}_0(\lambda-\tau,\eta - \xi) \langle \eta - \xi \rangle^{\frac{1}{2}+\epsilon}}{\langle C_+ \rangle^{\frac{1}{2}-}} 
\frac{|\xi|^{\frac{1}{2}-\epsilon}}{|\eta|^{\frac{1}{2}-\epsilon} |\eta - \xi|^{\frac{1}{2}-\epsilon}} \\
& \hspace{2em} \frac{\tilde{F}_1(\sigma,\zeta)}{\langle \zeta\rangle^{\frac{1}{2}+\epsilon}} 
\frac{\tilde{F}_2(\sigma - \tau, \zeta - \xi)}{\langle E_+ \rangle^{\frac{1}{2}+} \langle \zeta - \xi \rangle^{\frac{1}{2}+\epsilon}}
 \frac{|\xi|^{\frac{1}{2}}}{|\zeta|^{\frac{1}{2}} |\zeta - \xi|^{\frac{1}{2}}}
  d\sigma d\zeta d\tau d\xi d\eta d\lambda \\
& \lesssim \| ( \frac{\tilde{F}_3(\lambda,\eta)}{\langle \eta \rangle^{\frac{1}{2}+\epsilon}|\eta|^{\frac{1}{2}-\epsilon} \langle B_+ \rangle^{\epsilon}})\check{\,\,}\|_{L^{2+}_t L^{\infty}_x} 
\|( \frac{\tilde{F}_0(\lambda - \tau,\eta - \xi)}{\langle C_+\rangle^{\frac{1}{2}-}})\check{\,\,}\|_{L^{\infty -}_t L^2_x}\| F_1\|_{L^2_t L^{2}_x} \\
& \hspace{2em} \| ( \frac{\tilde{F}_2(\sigma - \tau,\zeta -\xi)}{\langle E_+ \rangle^{\frac{1}{2}+} \langle \zeta - \xi \rangle^{\frac{1}{2}+\epsilon} |\zeta - \xi|^{\frac{1}{2}}})\check{\,\,}\|_{L^{\infty}_t L^{\infty }_x}
\\
 & \lesssim \prod_{i=0}^3 \|F_i\|_{L^2_{xt}} \, .
\end{align*}
For the first factor we used the Sobolev estimate $ \|f\|_{L_t^2 L_x^{\infty}} \lesssim   \|f\|_{X_+^{1+,0}}$ and the estimate $\|f\|_{L_t^{4+} L_x^{\infty}} \lesssim \|f\|_{L_t^{4+} H_x^{\epsilon,\infty -}} \lesssim \|f\|_{X_+^{\frac{3}{4}+\epsilon +,\frac{1}{2}+}}$,which follows from Sobolev's embedding and Strichartz' estimate, so that an interpolation gives
\begin{equation}
\label{42}
\|f\|_{L_t^{2+} L_x^{\infty}} \lesssim  \\\|f\|_{X_+^{1,\epsilon}} \, ,
\end{equation}
which gives the desired bound for $|\eta| \ge 1$, whereas the case $|\eta| \le 1$ is easy.\\
{\bf Case 1.2:} $ \langle C_+ \rangle \ge |\xi| $.\\
\begin{align*}
I & \lesssim \int \frac{\tilde{F}_3(\lambda,\eta)}{\langle \eta \rangle^{\frac{1}{2}+\epsilon} \langle B_+ \rangle^{\frac{1}{2}+}} 
\frac{\tilde{F}_0(\lambda-\tau,\eta - \xi) \langle \eta - \xi \rangle^{\frac{1}{2}+\epsilon}}
{|\xi|^{\frac{1}{2}-}} 
\\
& \hspace{2em} \frac{\tilde{F}_1(\sigma,\zeta)}{\langle \zeta\rangle^{\frac{1}{2}+\epsilon}} 
\frac{\tilde{F}_2(\sigma - \tau, \zeta - \xi)}{\langle E_+ \rangle^{\frac{1}{2}+} \langle \zeta - \xi \rangle^{\frac{1}{2}+\epsilon}}
 \frac{|\xi|^{\frac{1}{2}}}{|\zeta|^{\frac{1}{2}} |\zeta - \xi|^{\frac{1}{2}}}
  d\sigma d\zeta d\tau d\xi d\eta d\lambda \, .
\end{align*}
Estimating $\langle \eta - \xi \rangle \le \langle \eta \rangle + \langle \xi \rangle$ we consider two different cases. \\
{\bf Case 1.2.1:} $ |\eta| \ge |\xi|$.\\
In this case we obtain
\begin{align*}
I & \lesssim \| ( \frac{\tilde{F}_3(\lambda,\eta)}{ \langle B_+ \rangle^{\frac{1}{2}+}})\check{\,\,}\|_{L^{\infty}_t L^2_x}
\| F_0\|_{L^2_t L^{2}_x}
\|( \frac{\tilde{F}_1(\sigma,\zeta )}{\langle \zeta \rangle^{\frac{1}{2}+\epsilon -} |\zeta|^{\frac{1}{2}})}\check{\,\,}\|_{L^2_t L^{\infty}_x} \\
& \hspace{2em} \| ( \frac{\tilde{F}_2(\sigma - \tau,\zeta -\xi)}{\langle E_+ \rangle^{\frac{1}{2}+} \langle \zeta - \xi \rangle^{\frac{1}{2}+\epsilon} |\zeta - \xi|^{\frac{1}{2}}})\check{\,\,}\|_{L^{\infty}_t L^{\infty }_x}
\\
 & \lesssim \prod_{i=0}^3 \|F_i\|_{L^2_{xt}} \, .
 \end{align*}
{\bf Case 1.2.2:} $ |\xi| \ge |\eta|$.\\
In this case we obtain
\begin{align*}
I & \lesssim \| ( \frac{\tilde{F}_3(\lambda,\eta)}{\langle \eta \rangle^{\frac{1}{2}+\epsilon} \langle B_+ \rangle^{\frac{1}{2}+}})\check{\,\,}\|_{L^{\infty}_t L^{4+}_x}
\| F_0\|_{L^2_t L^{2}_x}
\|( \frac{\tilde{F}_1(\sigma,\zeta )}{ |\zeta|^{\frac{1}{2}-}})\check{\,\,}\|_{L^2_t L^{4-}_x} \\
& \hspace{2em} \| ( \frac{\tilde{F}_2(\sigma - \tau,\zeta -\xi)}{\langle E_+ \rangle^{\frac{1}{2}+} \langle \zeta - \xi \rangle^{\frac{1}{2}+\epsilon} |\zeta - \xi|^{\frac{1}{2}}})\check{\,\,}\|_{L^{\infty}_t L^{\infty }_x}
\\
 & \lesssim \prod_{i=0}^3 \|F_i\|_{L^2_{xt}} \, .
 \end{align*}
{\bf Case 2:} $|B_+| \ge |A|,|C_+|$ and $|A| \ge |D_+|,|E_+|$. \\
{\bf Case 2.1:} $ |C_+| \le |\xi|$. \\
This case is treated as follows:
\begin{align*}
I & \lesssim \int \frac{\tilde{F}_3(\lambda,\eta)}{|\eta|^{\frac{1}{2}} \langle \eta \rangle^{\frac{1}{2}+\epsilon}} 
\frac{\tilde{F}_0(\lambda-\tau,\eta - \xi) \langle \eta - \xi \rangle^{\epsilon}}
{\langle C_+ \rangle^{\frac{1}{2}+}} |\xi|^{\frac{1}{2}+} 
\\
& \hspace{2em} |A|^{\frac{1}{2}} \frac{\tilde{F}_1(\sigma,\zeta)}{\langle \zeta\rangle^{\frac{1}{2}+\epsilon} \langle D_+ \rangle^{\frac{1}{2}+}} 
\frac{\tilde{F}_2(\sigma - \tau, \zeta - \xi)}{\langle E_+ \rangle^{\frac{1}{2}+} \langle \zeta - \xi \rangle^{\frac{1}{2}+\epsilon}}
 \frac{|\xi|^{\frac{1}{2}}}{|\zeta|^{\frac{1}{2}} |\zeta - \xi|^{\frac{1}{2}}}
  d\sigma d\zeta d\tau d\xi d\eta d\lambda \, .
\end{align*}
Estimating $\langle \eta - \xi \rangle^{\epsilon} \lesssim \langle \eta \rangle^{\epsilon} + \langle \xi \rangle^{\epsilon} $ we consider two different cases. \\
{\bf Case 2.1.1:} $|\eta| \gtrsim |\xi|$.\\
We obtain the bound
\begin{align*}
I & \lesssim \| ( \frac{\tilde{F}_3(\lambda,\eta)}{|\eta|^{\frac{1}{2}} \langle \eta \rangle^{\frac{1}{2}}})\check{\,\,}\|_{L^2_t L^{\infty -}_x}
\|( \frac{\tilde{F}_0(\lambda - \tau,\eta - \xi)}{\langle C_+ \rangle^{\frac{1}{2}+}})\check{\,\,}\|_{L^{\infty}_t L^2_x} \\
& \hspace{2em} \| (\frac{\tilde{F_1}(\sigma,\zeta)}{\langle \zeta \rangle^{\epsilon -} \langle D_+ \rangle^{\frac{1}{2}+}} \frac{\tilde{F}_2(\sigma - \tau,\zeta -\xi)}{\langle E_+ \rangle^{\frac{1}{2}+} \langle \zeta - \xi \rangle^{\frac{1}{2}+\epsilon} |\zeta - \xi|^{\frac{1}{2}}} ||\tau|-|\xi||^{\frac{1}{2}})\check{\,\,}\|_{L^2_t L^{2+}_x}
\\
 & \lesssim \prod_{i=0}^3 \|F_i\|_{L^2_{xt}} \, .
 \end{align*}
For the last factor we used Sobolev's embedding $H_x^{0+} \subset L_x^{2+}$ first and then Corollary \ref{Cor.1.1} with $\beta_0 =0+$ , $\beta_- = \frac{1}{2}$ , $\alpha_1=0+$ , $\alpha_2=1$.\\
{\bf Case 2.1.2:} $|\xi| \gg |\eta|$ ($\Rightarrow$ $|\xi| \sim |\eta - \xi|$). \\
We obtain
\begin{align*}
I & \lesssim \int \frac{\tilde{F}_3(\lambda,\eta)}{\langle \eta \rangle^{\frac{1}{2}+\epsilon} \langle B_+ \rangle^{\epsilon +}} 
\frac{\tilde{F}_0(\lambda-\tau,\eta - \xi) \langle \eta - \xi \rangle^{\frac{1}{2}+\epsilon}}
{\langle C_+ \rangle^{\frac{1}{2}-}} \frac{|\xi|^{\frac{1}{2}}}{|\eta|^{\frac{1}{2}} |\eta-\xi|^{\frac{1}{2}}}
\\
& \hspace{2em}  \frac{\tilde{F}_1(\sigma,\zeta)}{\langle \zeta\rangle^{\frac{1}{2}+\epsilon} \langle D_+ \rangle^{\frac{1}{2}+}} 
\frac{\tilde{F}_2(\sigma - \tau, \zeta - \xi)}{\langle E_+ \rangle^{\frac{1}{2}+} \langle \zeta - \xi \rangle^{\frac{1}{2}+\epsilon}}
 \frac{|\xi|^{\frac{1}{2}}}{|\zeta|^{\frac{1}{2}} |\zeta - \xi|^{\frac{1}{2}}}|A|^{\frac{1}{2}}
  d\sigma d\zeta d\tau d\xi d\eta d\lambda \\
 & \lesssim \| ( \frac{\tilde{F}_3(\lambda,\eta)}{|\eta|^{\frac{1}{2}} \langle \eta \rangle^{\frac{1}{2}+\epsilon}\langle B_+ \rangle^{\epsilon +}})\check{\,\,}\|_{L^{2+}_t L^{\infty}_x}
\|( \frac{\tilde{F}_0(\lambda - \tau,\eta - \xi)}{\langle C_+ \rangle^{\frac{1}{2}-}})\check{\,\,}\|_{L^{\infty -}_t L^2_x} \\
& \hspace{2em} \| (\frac{\tilde{F_1}(\sigma,\zeta)}{ \langle D_+ \rangle^{\frac{1}{2}+}} \frac{\tilde{F}_2(\sigma - \tau,\zeta -\xi)}{\langle E_+ \rangle^{\frac{1}{2}+} \langle \zeta - \xi \rangle^{\frac{1}{2}+\epsilon} |\zeta - \xi|^{\frac{1}{2}}} ||\tau|-|\xi||^{\frac{1}{2}})\check{\,\,}\|_{L^2_t L^2_x}
\\
 & \lesssim \prod_{i=0}^3 \|F_i\|_{L^2_{xt}} \, . 
\end{align*}
Here we used (\ref{42}) for the first factor and Cor. \ref{Cor.1.1} with $\beta_0=0$ , $\beta_- = \frac{1}{2}$ , $ \alpha_1 =0,$  $\alpha_2 = 1$ for the last factor. \\
{\bf Case 2.2:} $|C_+| \ge |\xi|$. \\
We obtain
\begin{align*}
I & \lesssim \int \frac{\tilde{F}_3(\lambda,\eta)}{\langle \eta \rangle^{\frac{1}{2}+\epsilon} \langle B_+ \rangle^{\frac{1}{2} +}} 
\frac{\tilde{F}_0(\lambda-\tau,\eta - \xi) \langle \eta - \xi \rangle^{\frac{1}{2}+\epsilon}}{|\xi|^{\frac{1}{2}-}}
\\
& \hspace{2em}  \frac{\tilde{F}_1(\sigma,\zeta)}{\langle \zeta\rangle^{\frac{1}{2}+\epsilon} \langle D_+ \rangle^{\frac{1}{2}+}} 
\frac{\tilde{F}_2(\sigma - \tau, \zeta - \xi)}{ \langle \zeta - \xi \rangle^{\frac{1}{2}+\epsilon} \langle E_+ \rangle^{\frac{1}{2}+}}
 \frac{|\xi|^{\frac{1}{2}}}{|\zeta|^{\frac{1}{2}} |\zeta - \xi|^{\frac{1}{2}}}|A|^{\frac{1}{2}}
  d\sigma d\zeta d\tau d\xi d\eta d\lambda
\end{align*}
We use our assumption $|\zeta| \ge |\zeta - \xi|$ , so that $|\xi| \lesssim |\zeta|$, and estimate $\langle \eta - \xi \rangle \le \langle \eta \rangle + \langle \xi \rangle$.\\
{\bf Case 2.2.1:} $|\eta| \ge |\xi|$.\\
In this case we obtain
\begin{align*}
I & \lesssim \| ( \frac{\tilde{F}_3(\lambda,\eta)}{\langle B_+ \rangle^{\frac{1}{2} +}})\check{\,\,}\|_{L^{\infty}_t L^2_x}
\|{F}_0\|_{L^2_t L^2_x} \\
& \hspace{2em} \| (\frac{\tilde{F_1}(\sigma,\zeta)}{\langle \zeta \rangle^{\frac{1}{2}+\epsilon -} |\zeta|^{\frac{1}{2}} \langle D_+ \rangle^{\frac{1}{2}+}} \frac{\tilde{F}_2(\sigma - \tau,\zeta -\xi)}{\langle E_+ \rangle^{\frac{1}{2}+} \langle \zeta - \xi \rangle^{\frac{1}{2}+\epsilon} |\zeta - \xi|^{\frac{1}{2}}} |A|^{\frac{1}{2}})\check{\,\,}\|_{L^2_t L^{\infty}_x}
\\
&\lesssim \|F_3\|_{L^2_{xt}} \|F_0\|_{L^2_{xt}} \\
& \hspace{2em} \| ((|\xi|^{1-}+|\xi|^{1+})\frac{\tilde{F_1}}{\langle \zeta \rangle^{\frac{1}{2}+\epsilon -} |\zeta|^{\frac{1}{2}} \langle D_+ \rangle^{\frac{1}{2}+}} \frac{\tilde{F}_2(\sigma - \tau,\zeta -\xi)}{\langle E_+ \rangle^{\frac{1}{2}+} \langle \zeta - \xi \rangle^{\frac{1}{2}+\epsilon} |\zeta - \xi|^{\frac{1}{2}}} |A|^{\frac{1}{2}})\check{\,\,}\|_{L^2_t L^2_x}
\\
\\
&\lesssim \|F_3\|_{L^2_{xt}} \|F_0\|_{L^2_{xt}} \\
& \hspace{2em} \| (\frac{\tilde{F_1}}{\langle \zeta \rangle^{\epsilon -} \langle D_+ \rangle^{\frac{1}{2}+}} \frac{\tilde{F}_2(\sigma - \tau,\zeta -\xi)}{\langle E_+ \rangle^{\frac{1}{2}+} \langle \zeta - \xi \rangle^{\frac{1}{2}+\epsilon} |\zeta - \xi|^{\frac{1}{2}}} ||\tau|-|\xi||^{\frac{1}{2}})\check{\,\,}\|_{L^2_t L^2_x} \,.
\end{align*}
For the last factor we applied Sobolev's embedding and Cor. \ref{Cor.1.1} with $\beta_0 =0$ , $ \beta_- = \frac{1}{2}$ , $\alpha_1 =0$ , $ \alpha_2 =1$ to obtain the bound $\|F_1\|_{L^2_{xt}}\|F_2\|_{L^2_{xt}}$ for it.\\
{\bf Case 2.2.2:} $|\xi| \ge |\eta|$ ($\Rightarrow$ $|\eta-\xi| \lesssim |\xi| \lesssim |\zeta|$).\\
We arrive at
\begin{align*}
I & \lesssim \| ( \frac{\tilde{F}_3(\lambda,\eta)}{\langle \eta \rangle^{\frac{1}{2}+\epsilon} \langle B_+ \rangle^{\frac{1}{2} +}})\check{\,\,}\|_{L^{\infty}_t L^{4++}_x}
\|{F}_0\|_{L^2_t L^2_x} \\
& \hspace{2em} \| (\frac{\tilde{F_1}(\sigma,\zeta)}{ |\zeta|^{\frac{1}{2}-} \langle D_+ \rangle^{\frac{1}{2}+}} \frac{\tilde{F}_2(\sigma - \tau,\zeta -\xi)}{\langle E_+ \rangle^{\frac{1}{2}+} \langle \zeta - \xi \rangle^{\frac{1}{2}+\epsilon} |\zeta - \xi|^{\frac{1}{2}}} |A|^{\frac{1}{2}})\check{\,\,}\|_{L^2_t L^{4--}_x} \, .
\end{align*}
We estimate the first factor using Sobolev by $\|F_3\|_{L^2_{xt}}$ and the last factor by Sobolev and Cor. \ref{Cor.1.1} with $\beta_0 = 0$ , $\beta_- = \frac{1}{2}$ , $ \alpha_1 = 0+$ , $\alpha_2 = 1-$ by
\begin{align*}
&\| (|\xi|^{\frac{1}{2}--}
\frac{\tilde{F_1}(\sigma,\zeta)}{ |\zeta|^{\frac{1}{2}-} \langle D_+ \rangle^{\frac{1}{2}+}} 
\frac{\tilde{F}_2(\sigma - \tau,\zeta -\xi)}{\langle E_+ \rangle^{\frac{1}{2}+} \langle \zeta - \xi \rangle^{\frac{1}{2}+\epsilon} |\zeta - \xi|^{\frac{1}{2}}} |A|^{\frac{1}{2}})\check{\,\,}\|_{L^2_t L^2_x}\\
&\lesssim \| (\frac{\tilde{F_1}(\sigma,\zeta)}{ |\zeta|^{0+} \langle D_+ \rangle^{\frac{1}{2}+}} \frac{\tilde{F}_2(\sigma - \tau,\zeta -\xi)}{\langle E_+ \rangle^{\frac{1}{2}+} \langle \zeta - \xi \rangle^{\frac{1}{2}+\epsilon} |\zeta - \xi|^{\frac{1}{2}}} |A|^{\frac{1}{2}})\check{\,\,}\|_{L^2_t L^2_x}\\
& \lesssim \|F_1\|_{L^2_{xt}} \|F_2\|_{L^2_{xt}} \, .
\end{align*}
{\bf Case 3:} $|A| \ge |B_+|,|C_+|$ and $ |D_+| \ge |A|,|E_+|$. \\
{\bf Case 3.1:} $|C_+| \le |\eta|$.\\
Using $|\xi| \le |\zeta| + |\zeta - \xi| \lesssim |\zeta|$ we obtain
\begin{align*}
I & \lesssim \int \frac{\tilde{F}_3(\lambda,\eta)}{ \langle \eta \rangle^{\frac{1}{2}+\epsilon} \langle B_+ \rangle^{\frac{1}{2}+}} 
\frac{\tilde{F}_0(\lambda-\tau,\eta - \xi) \langle \eta - \xi \rangle^{\frac{1}{2}+\epsilon} |\eta|^{0+}}
{\langle C_+ \rangle^{\frac{1}{2}+}} |A|^{\frac{1}{2}} \frac{|\xi|^{\frac{1}{2}}}{|\eta|^{\frac{1}{2}} |\eta - \xi|^{\frac{1}{2}}} 
\\
& \hspace{2em} \frac{\tilde{F}_1(\sigma,\zeta)}{\langle \zeta\rangle^{\frac{1}{2}+\epsilon}} 
\frac{\tilde{F}_2(\sigma - \tau, \zeta - \xi)}{\langle E_+ \rangle^{\frac{1}{2}+} \langle \zeta - \xi \rangle^{\frac{1}{2}+\epsilon}}
 \frac{|\xi|^{\frac{1}{2}}}{|\zeta|^{\frac{1}{2}} |\zeta - \xi|^{\frac{1}{2}}}
  d\sigma d\zeta d\tau d\xi d\eta d\lambda 
\\
&\lesssim \int \frac{\tilde{F}_3(\lambda,\eta)}{ \langle \eta \rangle^{\frac{1}{2}+\epsilon} |\eta|^{\frac{1}{2}-} \langle B_+ \rangle^{\frac{1}{2}+}} 
\frac{\tilde{F}_0(\lambda-\tau,\eta - \xi) \langle \eta - \xi \rangle^{\epsilon}}
{\langle C_+ \rangle^{\frac{1}{2}+}} |A|^{\frac{1}{2}}
\\
& \hspace{2em} \frac{\tilde{F}_1(\sigma,\zeta)}{\langle \zeta\rangle^{\epsilon}} 
\frac{\tilde{F}_2(\sigma - \tau, \zeta - \xi)}{\langle E_+ \rangle^{\frac{1}{2}+} \langle \zeta - \xi \rangle^{\frac{1}{2}+\epsilon} |\zeta - \xi|^{\frac{1}{2}}}
  d\sigma d\zeta d\tau d\xi d\eta d\lambda  \, .
\end{align*}
Now we have $\langle \eta - \xi \rangle^{\epsilon} \lesssim \langle \eta \rangle^{\epsilon} + \langle \xi \rangle^{\epsilon}$.\\
{\bf Case 3.1.1:} $|\eta| \ge |\xi|$.\\
We arrive at
\begin{align*}
I & \lesssim \int \frac{\tilde{F}_3(\lambda,\eta)}{ \langle \eta \rangle^{\frac{1}{2}} |\eta|^{\frac{1}{2}-} \langle B_+ \rangle^{\frac{1}{2}+}} 
\frac{\tilde{F}_0(\lambda-\tau,\eta - \xi)}
{\langle C_+ \rangle^{\frac{1}{2}+}} |A|^{\frac{1}{2}} |\xi|^{-\epsilon}
\\
& \hspace{2em} |\xi|^{\epsilon} \frac{\tilde{F}_1(\sigma,\zeta)}{\langle \zeta\rangle^{\epsilon}} 
\frac{\tilde{F}_2(\sigma - \tau, \zeta - \xi)}{\langle E_+ \rangle^{\frac{1}{2}+} \langle \zeta - \xi \rangle^{\frac{1}{2}+\epsilon} |\zeta - \xi|^{\frac{1}{2}}}
  d\sigma d\zeta d\tau d\xi d\eta d\lambda \\
 & \lesssim \| ( \frac{\tilde{F}_3(\lambda,\eta)}{\langle \eta \rangle^{\frac{1}{2}} |\eta|^{\frac{1}{2}-} \langle B_+ \rangle^{\frac{1}{2}+}} \frac{\tilde{F}_0(\lambda - \tau,\eta - \xi)}{\langle C_+\rangle^{\frac{1}{2}+}} |A|^{\frac{1}{2}} |\xi|^{-\epsilon})\check{\,\,}\|_{L^2_{xt}} \|F_1\|_{L^2_{xt}} \\
 & \hspace{2em} \|( \frac{\tilde{F}_2(\sigma - \tau,\zeta -\xi)}{\langle E_+ \rangle^{\frac{1}{2}+} \langle \zeta - \xi \rangle^{\frac{1}{2}+\epsilon} |\zeta - \xi|^{\frac{1}{2}}} )\check{\,\,}\|_{L^{\infty}_{xt}} \\
 & \lesssim \prod_{i=0}^3 \|F_i\|_{L^2_{xt}} \, ,
\end{align*}  
where we used Cor. \ref{Cor.1.1} with $\beta_0 = -\epsilon$ , $ \beta_- = \frac{1}{2}$ , $\alpha_1 = 1-\epsilon$ , $\alpha_2 =0$ for the first factor. \\
{\bf Case 3.1.2:} $|\xi| \ge |\eta|$.\\
An application of Cor. \ref{Cor.1.1} with $\beta_0 =0$ , $\beta_- = \frac{1}{2}$ , $\alpha_1=1$ , $\alpha_2 =0$ gives the estimate
\begin{align*}
I & \lesssim \| ( \frac{\tilde{F}_3(\lambda,\eta)}{\langle \eta \rangle^{\frac{1}{2}+\epsilon} |\eta|^{\frac{1}{2}-} \langle B_+ \rangle^{\frac{1}{2}+}} \frac{\tilde{F}_0(\lambda - \tau,\eta - \xi)}{\langle C_+\rangle^{\frac{1}{2}+}} |A|^{\frac{1}{2}} )\check{\,\,}\|_{L^2_{xt}} \|F_1\|_{L^2_{xt}} \\
 & \hspace{2em} \|( \frac{\tilde{F}_2(\sigma - \tau,\zeta -\xi)}{\langle E_+ \rangle^{\frac{1}{2}+} \langle \zeta - \xi \rangle^{\frac{1}{2}+\epsilon} |\zeta - \xi|^{\frac{1}{2}}} )\check{\,\,}\|_{L^{\infty}_{xt}} \\
 & \lesssim \prod_{i=0}^3 \|F_i\|_{L^2_{xt}} \, .
\end{align*}
{\bf Case 3.2:} $|C_+| \ge |\eta|$.\\
We obtain using again our tacid assumption $|\zeta| \ge |\zeta - \xi|$:
\begin{align*}
I & \lesssim \int \frac{\tilde{F}_3(\lambda,\eta)}{ \langle \eta \rangle^{\frac{1}{2}+\epsilon} \langle B_+ \rangle^{\frac{1}{2}+}} 
\frac{\tilde{F}_0(\lambda-\tau,\eta - \xi) \langle \eta - \xi \rangle^{\frac{1}{2}+\epsilon}}
{|\eta|^{\frac{1}{2}-}}
\\
& \hspace{2em}  \frac{\tilde{F}_1(\sigma,\zeta)}{\langle \zeta\rangle^{\frac{1}{2}+\epsilon}} 
\frac{\tilde{F}_2(\sigma - \tau, \zeta - \xi)}{\langle E_+ \rangle^{\frac{1}{2}+} \langle \zeta - \xi \rangle^{\frac{1}{2}+\epsilon} |\zeta - \xi|^{\frac{1}{2}}}
  d\sigma d\zeta d\tau d\xi d\eta d\lambda \, .
\end{align*}
{\bf Case 3.2.1:} $|\eta| \ge |\xi|$.\\
We estimate as follows:
\begin{align*}
I & \lesssim \| ( \frac{\tilde{F}_3(\lambda,\eta)}{ |\eta|^{\frac{1}{2}-} \langle B_+ \rangle^{\frac{1}{2}+}} )\check{\,\,}\|_{L^{\infty}_t L^{4-}_x} \|F_0\|_{L^2_{xt}} \| (\frac{\tilde{F}_1(\sigma,\zeta )}{\langle \zeta \rangle^{\frac{1}{2}+\epsilon}})\check{\,\,}\|_{L^2_t L^{4+}_x} \\
 & \hspace{2em} \|( \frac{\tilde{F}_2(\sigma - \tau,\zeta -\xi)}{\langle E_+ \rangle^{\frac{1}{2}+} \langle \zeta - \xi \rangle^{\frac{1}{2}+\epsilon} |\zeta - \xi|^{\frac{1}{2}}} )\check{\,\,}\|_{L^{\infty}_{xt}} \\
 & \lesssim \prod_{i=0}^3 \|F_i\|_{L^2_{xt}} \, .
\end{align*}
{\bf Case 3.2.2:} $|\eta| \le |\xi|$.\\
In this case we obtain
\begin{align*}
I & \lesssim \| ( \frac{\tilde{F}_3(\lambda,\eta)}{\langle \eta \rangle^{\frac{1}{2}+\epsilon} |\eta|^{\frac{1}{2}-} \langle B_+ \rangle^{\frac{1}{2}+}} )\check{\,\,}\|_{L^{\infty}_{xt}} \|F_0\|_{L^2_{xt}} \|F_1\|_{L^2_{xt}}  \\
 & \hspace{2em} \|( \frac{\tilde{F}_2(\sigma - \tau,\zeta -\xi)}{\langle E_+ \rangle^{\frac{1}{2}+} \langle \zeta - \xi \rangle^{\frac{1}{2}+\epsilon} |\zeta - \xi|^{\frac{1}{2}}} )\check{\,\,}\|_{L^{\infty}_{xt}} \\
 & \lesssim \prod_{i=0}^3 \|F_i\|_{L^2_{xt}} \, .
\end{align*}
{\bf Case 4:} $|C_+| \ge |A|,|B_+|$ and $|D_+| \ge |A|,|E_+|$.\\
Using $ \Theta_{+,+} \lesssim \frac{|\xi|^{\frac{1}{2}} |\eta|^{0+}}{|\eta|^{\frac{1}{2}} |\eta - \xi|^{\frac{1}{2}}} |C_+|^{\frac{1}{2}-}$ and $|\xi| \le |\zeta| + |\zeta - \xi| \lesssim |\zeta|$ we obtain
\begin{align*}
I & \lesssim \int \frac{\tilde{F}_3(\lambda,\eta)}{ \langle \eta \rangle^{\frac{1}{2}+\epsilon} \langle B_+ \rangle^{\frac{1}{2}+}} 
\tilde{F}_0(\lambda-\tau,\eta - \xi) \langle \eta - \xi \rangle^{\frac{1}{2}+\epsilon}
\frac{|\xi|^{\frac{1}{2}} |\eta|^{0+}}{|\eta|^{\frac{1}{2}} |\eta - \xi|^{\frac{1}{2}}}
\\
& \hspace{2em}  \frac{\tilde{F}_1(\sigma,\zeta)}{\langle \zeta\rangle^{\frac{1}{2}+\epsilon}} 
\frac{\tilde{F}_2(\sigma - \tau, \zeta - \xi)}{\langle E_+ \rangle^{\frac{1}{2}+} \langle \zeta - \xi \rangle^{\frac{1}{2}+\epsilon} |\zeta - \xi|^{\frac{1}{2}}}
  d\sigma d\zeta d\tau d\xi d\eta d\lambda \, .
\end{align*}
{\bf Case 4.1:} $|\xi| \ll |\eta - \xi| \sim |\eta|$.\\
We conclude
\begin{align*}
I & \lesssim \| ( \frac{\tilde{F}_3(\lambda,\eta)}{ |\eta|^{1-} \langle B_+ \rangle^{\frac{1}{2}+}} )\check{\,\,}\|_{L^{\infty}_t L^{\infty-}_x} \|F_0\|_{L^2_{xt}} \| (\frac{\tilde{F}_1(\sigma,\zeta )}{\langle \zeta \rangle^{\epsilon}})\check{\,\,}\|_{L^2_t L^{2+}_x} \\
 & \hspace{2em} \|( \frac{\tilde{F}_2(\sigma - \tau,\zeta -\xi)}{\langle E_+ \rangle^{\frac{1}{2}+} \langle \zeta - \xi \rangle^{\frac{1}{2}+\epsilon} |\zeta - \xi|^{\frac{1}{2}}} )\check{\,\,}\|_{L^{\infty}_{xt}} \\
 & \lesssim \prod_{i=0}^3 \|F_i\|_{L^2_{xt}} \, .
\end{align*}
{\bf Case 4.2:} $|\xi| \gtrsim |\eta|$ ($\Rightarrow$ $|\eta - \xi| \lesssim |\xi| \lesssim |\zeta|$).\\
Similarly as before we obtain
\begin{align*}
I & \lesssim \| ( \frac{\tilde{F}_3(\lambda,\eta)}{ |\eta|^{\frac{1}{2}-} \langle \eta \rangle^{\frac{1}{2}+\epsilon} \langle B_+ \rangle^{\frac{1}{2}+}} )\check{\,\,}\|_{L^{\infty}_{xt}} \|F_0\|_{L^2_{xt}} \|F_1\|_{L^2_{xt}}
 \\& \hspace{2em}
  \|( \frac{\tilde{F}_2(\sigma - \tau,\zeta -\xi)}{\langle E_+ \rangle^{\frac{1}{2}+} \langle \zeta - \xi \rangle^{\frac{1}{2}+\epsilon} |\zeta - \xi|^{\frac{1}{2}}} )\check{\,\,}\|_{L^{\infty}_{xt}} \\
 & \lesssim \prod_{i=0}^3 \|F_i\|_{L^2_{xt}} \, .
\end{align*}
{\bf Case 5:} $|C_+| \ge |A|,|B_+|$ and $|A| \ge |D_+|,|E_+|$.\\
In this case we estimate as follows:
\begin{align*}
I & \lesssim \int \frac{\tilde{F}_3(\lambda,\eta)}{ \langle \eta \rangle^{\frac{1}{2}+\epsilon} \langle B_+ \rangle^{\frac{1}{2}+}} 
\tilde{F}_0(\lambda-\tau,\eta - \xi) \langle \eta - \xi \rangle^{\frac{1}{2}+\epsilon}
\frac{|\xi|^{\frac{1}{2}} |\eta|^{0+}}{|\eta|^{\frac{1}{2}} |\eta - \xi|^{\frac{1}{2}}}
\\
& \hspace{2em}  \frac{\tilde{F}_1(\sigma,\zeta)}{\langle \zeta\rangle^{\frac{1}{2}+\epsilon} \langle D_+ \rangle^{\frac{1}{2}+}}
\frac{\tilde{F}_2(\sigma - \tau, \zeta - \xi)}{\langle E_+ \rangle^{\frac{1}{2}+} \langle \zeta - \xi \rangle^{\frac{1}{2}+\epsilon} |\zeta - \xi|^{\frac{1}{2}}}||\tau|-|\xi||^{\frac{1}{2}}
  d\sigma d\zeta d\tau d\xi d\eta d\lambda \, .
\end{align*}
Estimating $\langle \eta - \xi \rangle^{\epsilon} \lesssim \langle \xi \rangle^{\epsilon} + \langle \eta \rangle^{\epsilon}$ we consider two subcases.\\
{\bf Case 5.1:} $|\xi| \ge |\eta|$ ($\Rightarrow$ $ \langle \eta - \xi \rangle^{\epsilon} \lesssim \langle \zeta \rangle^{\epsilon}$).\\
Thus we obtain
\begin{align*}
I & \lesssim \| ( \frac{\tilde{F}_3(\lambda,\eta)}{ |\eta|^{\frac{1}{2}-} \langle \eta \rangle^{\frac{1}{2}+\epsilon} \langle B_+ \rangle^{\frac{1}{2}+}} )\check{\,\,}\|_{L^{\infty}_{xt}} \|F_0\|_{L^2_{xt}} \\
 & \hspace{2em} \|(\frac{\tilde{F}_1(\sigma,\zeta )}{\langle D_+ \rangle^{\frac{1}{2}+}} \frac{\tilde{F}_2(\sigma - \tau,\zeta -\xi)}{\langle E_+ \rangle^{\frac{1}{2}+} \langle \zeta - \xi \rangle^{\frac{1}{2}+\epsilon} |\zeta - \xi|^{\frac{1}{2}}} ||\tau|-|\xi||^{\frac{1}{2}} )\check{\,\,}\|_{L^2_{xt}} \\
 & \lesssim \prod_{i=0}^3 \|F_i\|_{L^2_{xt}} \, ,
\end{align*}
applying Cor. \ref{Cor.1.1} with $\beta_0=0$ , $\beta_- = \frac{1}{2}$ , $\alpha_1 =0$ , $\alpha_2=1$. \\
{\bf Case 5.2:} $|\eta| \ge |\xi|$. \\
In this case we arrive at
\begin{align*}
I & \lesssim \| ( \frac{\tilde{F}_3(\lambda,\eta)}{ |\eta|^{\frac{1}{2}-} \langle \eta \rangle^{\frac{1}{2}} \langle B_+ \rangle^{\frac{1}{2}+}} )\check{\,\,}\|_{L^{\infty}_t L^{\infty -}_x} \|F_0\|_{L^2_{xt}} \\
 & \hspace{2em} \|(\frac{\tilde{F}_1(\sigma,\zeta )}{\langle \zeta \rangle^{\epsilon} \langle D_+ \rangle^{\frac{1}{2}+}} \frac{\tilde{F}_2(\sigma - \tau,\zeta -\xi)}{\langle E_+ \rangle^{\frac{1}{2}+} \langle \zeta - \xi \rangle^{\frac{1}{2}+\epsilon} |\zeta - \xi|^{\frac{1}{2}}} ||\tau|-|\xi||^{\frac{1}{2}} )\check{\,\,}\|_{L^2_t L^{2+}_x} \\
 & \lesssim \prod_{i=0}^3 \|F_i\|_{L^2_{xt}} \, ,
\end{align*}
where we used the Sobolev embedding $\dot{H}^{0+}_x  \subset L^{2+}_x$ for the last factor and then Cor. \ref{Cor.1.1} with $\beta_0=0+$ , $\beta_- = \frac{1}{2}$ , $\alpha_1 =0+$ , $\alpha_2=1$. \\
{\bf Case 6:} $|A| \ge |C_+|,|B_+|$ and $|A| \ge |D_+|,|E_+|$. \\
{\bf Case 6.1:} $|C_+| \le |\eta|$. \\
We estimate
\begin{align*}
I & \lesssim \int \frac{\tilde{F}_3(\lambda,\eta)}{ \langle \eta \rangle^{\frac{1}{2}+\epsilon} \langle B_+ \rangle^{\frac{1}{2}+}} 
\frac{\tilde{F}_0(\lambda-\tau,\eta - \xi) \langle \eta - \xi \rangle^{\frac{1}{2}+\epsilon}}{\langle C_+ \rangle^{\frac{1}{2}-}}
\frac{|\xi|^{\frac{1}{2}}}{|\eta|^{\frac{1}{2}} |\eta - \xi|^{\frac{1}{2}}} ||\tau|-|\xi||^{\frac{1}{2}}
\\
& \hspace{2em}  \frac{\tilde{F}_1(\sigma,\zeta)}{\langle \zeta \rangle^{\frac{1}{2}+\epsilon} \langle D_+ \rangle^{\frac{1}{2}+}}
\frac{\tilde{F}_2(\sigma - \tau, \zeta - \xi)}{\langle E_+ \rangle^{\frac{1}{2}+} \langle \zeta - \xi \rangle^{\frac{1}{2}+\epsilon}} \frac{|\xi|^{\frac{1}{2}}}{|\zeta|^{\frac{1}{2}} |\zeta - \xi|^{\frac{1}{2}}}||\tau|-|\xi||^{\frac{1}{2}}
  d\sigma d\zeta d\tau d\xi d\eta d\lambda \\
& \lesssim \int \frac{\tilde{F}_3(\lambda,\eta)}{ \langle \eta \rangle^{\frac{1}{2}+\epsilon} |\eta|^{\frac{1}{2}-} \langle B_+ \rangle^{\frac{1}{2}+}} 
\frac{\tilde{F}_0(\lambda-\tau,\eta - \xi) \langle \eta - \xi \rangle^{\epsilon} |\xi|^{\frac{1}{2}}}{\langle C_+ \rangle^{\frac{1}{2}+}}
 ||\tau|-|\xi||^{\frac{1}{2}}
\\
& \hspace{2em}  \frac{\tilde{F}_1(\sigma,\zeta)}{\langle \zeta \rangle^{\frac{1}{2}+\epsilon} \langle D_+ \rangle^{\frac{1}{2}+}}
\frac{\tilde{F}_2(\sigma - \tau, \zeta - \xi)}{\langle E_+ \rangle^{\frac{1}{2}+} \langle \zeta - \xi \rangle^{\frac{1}{2}+\epsilon} |\zeta - \xi|^{\frac{1}{2}}}||\tau|-|\xi||^{\frac{1}{2}}
  d\sigma d\zeta d\tau d\xi d\eta d\lambda \, .
\end{align*}
{\bf Case 6.1.1:} $|\eta| \ge |\xi|$. \\
We obtain
\begin{align*}
I & \lesssim \int \left(\frac{\tilde{F}_3(\lambda,\eta)}{ \langle \eta \rangle^{\frac{1}{2}} |\eta|^{\frac{1}{2}-} \langle B_+ \rangle^{\frac{1}{2}+}} 
\frac{\tilde{F}_0(\lambda-\tau,\eta - \xi)}{\langle C_+ \rangle^{\frac{1}{2}+}}
 ||\tau|-|\xi||^{\frac{1}{2}} |\xi|^{0-} \right)
\\
& \hspace{2em}  \left(\frac{\tilde{F}_1(\sigma,\zeta)}{\langle \zeta \rangle^{\epsilon -} \langle D_+ \rangle^{\frac{1}{2}+}}
\frac{\tilde{F}_2(\sigma - \tau, \zeta - \xi)}{\langle E_+ \rangle^{\frac{1}{2}+} \langle \zeta - \xi \rangle^{\frac{1}{2}+\epsilon} |\zeta - \xi|^{\frac{1}{2}}}||\tau|-|\xi||^{\frac{1}{2}} \right)
  d\sigma d\zeta d\tau d\xi d\eta d\lambda \, .
\end{align*}
We estimate both factors in $L^2_{xt}$ using Cor. \ref{Cor.1.1} with $\beta_0 = 0-$ , $\beta_- = \frac{1}{2}$ , $\alpha_1 = 1-,$  $ \alpha_2 =0$ and $\beta_0= 0$ , $ \beta_- = \frac{1}{2}$ , $\alpha_1 =0$ , $\alpha_2 =1$ , respectively. \\
{\bf Case 6.1.2:} $|\eta| \le |\xi|$.\\
We obtain
\begin{align*}
I & \lesssim \int \left(\frac{\tilde{F}_3(\lambda,\eta)}{ \langle \eta \rangle^{\frac{1}{2}+\epsilon} |\eta|^{\frac{1}{2}-} \langle B_+ \rangle^{\frac{1}{2}+}} 
\frac{\tilde{F}_0(\lambda-\tau,\eta - \xi)}{\langle C_+ \rangle^{\frac{1}{2}+}}
 ||\tau|-|\xi||^{\frac{1}{2}} \right)
\\
& \hspace{2em}  \left(\frac{\tilde{F}_1(\sigma,\zeta)}{ \langle D_+ \rangle^{\frac{1}{2}+}}
\frac{\tilde{F}_2(\sigma - \tau, \zeta - \xi)}{\langle E_+ \rangle^{\frac{1}{2}+} \langle \zeta - \xi \rangle^{\frac{1}{2}+\epsilon} |\zeta - \xi|^{\frac{1}{2}}}||\tau|-|\xi||^{\frac{1}{2}} \right)
  d\sigma d\zeta d\tau d\xi d\eta d\lambda \, .
\end{align*}
As in case 6.1.1 we estimate both factors in $L^2_{xt}$ using Cor. \ref{Cor.1.1} with $\beta_0 = 0$ , $\beta_- = \frac{1}{2}$ , $\alpha_1 = 1,$  $ \alpha_2 =0$ and $\beta_0= 0$ , $ \beta_- = \frac{1}{2}$ , $\alpha_1 =0$ , $\alpha_2 =1$ , respectively. \\
{\bf Case 6.2:} $ |C_+| \ge |\eta|$.\\
In this case we have
\begin{align*}
I & \lesssim \int \frac{\tilde{F}_3(\lambda,\eta)}{ \langle \eta \rangle^{\frac{1}{2}+\epsilon} \langle B_+ \rangle^{\frac{1}{2}+}} 
\frac{\tilde{F}_0(\lambda-\tau,\eta - \xi) \langle \eta - \xi \rangle^{\frac{1}{2}+\epsilon}}{|\eta|^{\frac{1}{2}-}}
\\
& \hspace{2em}  \frac{\tilde{F}_1(\sigma,\zeta)}{\langle \zeta \rangle^{\frac{1}{2} +\epsilon} \langle D_+ \rangle^{\frac{1}{2}+}}
\frac{\tilde{F}_2(\sigma - \tau, \zeta - \xi)}{\langle E_+ \rangle^{\frac{1}{2}+} \langle \zeta - \xi \rangle^{\frac{1}{2}+\epsilon} |\zeta - \xi|^{\frac{1}{2}}}||\tau|-|\xi||^{\frac{1}{2}} 
  d\sigma d\zeta d\tau d\xi d\eta d\lambda \, .
\end{align*}
{\bf Case 6.2.1:} $|\eta| \ge |\xi|$.\\
We arrive at the bound
\begin{align*}
I & \lesssim \| ( \frac{\tilde{F}_3(\lambda,\eta)}{ |\eta|^{\frac{1}{2}-} \langle B_+ \rangle^{\frac{1}{2}+}} )\check{\,\,}\|_{L^{\infty}_t L^{4-}_x} \|F_0\|_{L^2_{xt}} \\
 & \hspace{2em} \|(\frac{\tilde{F}_1(\sigma,\zeta )}{\langle \zeta \rangle^{\frac{1}{2}+\epsilon} \langle D_+ \rangle^{\frac{1}{2}+}} \frac{\tilde{F}_2(\sigma - \tau,\zeta -\xi)}{\langle E_+ \rangle^{\frac{1}{2}+} \langle \zeta - \xi \rangle^{\frac{1}{2}+\epsilon} |\zeta - \xi|^{\frac{1}{2}}} ||\tau|-|\xi||^{\frac{1}{2}} )\check{\,\,}\|_{L^2_t L^{4+}_x} \, .
\end{align*}
By Sobolev the first factor is estimated by $\|F_3\|_{L^2_{xt}}$, and the last factor by Sobolev`s embedding $\dot{H}^{\frac{1}{2}+}_x \subset L^{4+}_x$ followed by an application of Cor. \ref{Cor.1.1} with $\beta_0 = \frac{1}{2}+$ , $\beta_- = \frac{1}{2}$ , $\alpha_1 = \frac{1}{2}+$ , $\alpha_2 = 1$ , which gives the desired bound. \\
{\bf Case 6.2.2:} $ |\eta| \le |\xi| $. \\
We end up with the bound
\begin{align*}
I & \lesssim \| ( \frac{\tilde{F}_3(\lambda,\eta)}{\langle \eta \rangle^{\frac{1}{2}+\epsilon} |\eta|^{\frac{1}{2}-} \langle B_+ \rangle^{\frac{1}{2}+}} )\check{\,\,}\|_{L^{\infty}_t L^{\infty}_x} \|F_0\|_{L^2_{xt}} \\
 & \hspace{2em} \|(\frac{\tilde{F}_1(\sigma,\zeta )}{ \langle D_+ \rangle^{\frac{1}{2}+}} \frac{\tilde{F}_2(\sigma - \tau,\zeta -\xi)}{\langle E_+ \rangle^{\frac{1}{2}+} \langle \zeta - \xi \rangle^{\frac{1}{2}+\epsilon} |\zeta - \xi|^{\frac{1}{2}}} ||\tau|-|\xi||^{\frac{1}{2}} )\check{\,\,}\|_{L^2_{xt}} \, .
\end{align*} 
Cor. \ref{Cor.1.1} with $\beta_0 =0$ , $\beta_- = \frac{1}{2}$ , $\alpha_1 =0$ , $\alpha_1 = 1$ implies the desired bound.

This completes the proof of Part I, where all the signs are $+$ -signs. \\
{\bf Part II:} Next we consider the case $\pm_3 = +$ , $ \pm = -$ and $\pm_1 = +$ , $\pm_2 = -$. In the same way all the cases can be treated where $\pm$ and $\pm_3$ as well as $\pm_1$
and $\pm_2$ have different signs. 

We use the following estimates which immediately follow from Lemma \ref{Lemma2.2}:
\begin{align}
\label{48}
\Theta_{-,+} & \lesssim \frac{(|\eta|+|\eta-\xi|)^{\frac{1}{2}}}{|\eta|^{\frac{1}{2}} |\eta - \xi|^{\frac{1}{2}}} (|A|^{\frac{1}{2}} + |B_+|^{\frac{1}{2}} + \min(|\eta|,|\eta - \xi|)^{0+} |C_-|^{\frac{1}{2}-} ) \\ \label{49}
\overline{\Theta}_{+,-} & \lesssim \frac{(|\zeta|+|\zeta - \xi|)^{\frac{1}{2}}}{|\zeta|^{\frac{1}{2}} |\zeta - \xi|^{\frac{1}{2}}} (|A|^{\frac{1}{2}} + |D_+|^{\frac{1}{2}} + |E_-|^{\frac{1}{2}})
\end{align}
We first make the important remark that we may assume in all the cases where one has different signs that concerning $\Theta_{-,+}$:
\begin{equation}
\label{50}
 |\xi| \ll |\eta| \sim |\eta - \xi| 
\end{equation}
and similarly concerning  $\overline{\Theta}_{+,-}$:
\begin{equation}
\label{51}
|\xi| \ll |\zeta| \sim |\zeta - \xi| \, .
\end{equation}
If one namely has $|\eta| \ll |\eta - \xi|$, then $|\xi| \sim |\eta-\xi|$, and thus
the factor $\frac{(|\eta|+|\eta-\xi|)^{\frac{1}{2}}}{|\eta|^{\frac{1}{2}} |\eta - \xi|^{\frac{1}{2}}}$ is equivalent to $\frac{|\xi|^{\frac{1}{2}}}{|\eta|^{\frac{1}{2}} |\eta - \xi|^{\frac{1}{2}}}$. If $|\eta| \gg |\eta-\xi|$, then $|\xi| \sim |\eta|$, and the same is true, and also in the case $|\xi| \sim |\eta| \sim |\eta - \xi|$. Thus in all these cases we have the same estimate for $\Theta_{-,+}$ as for $\Theta_{+,+}$ , especially the estimates (\ref{45}) and (\ref{46}) with $C_+$ replaced by $C_-$, so the same arguments in this case hold true, if  (\ref{50}) is violated. The same arguments work for $\overline{\Theta}_{+,-}$, especially (\ref{45'}) with $E_+$ replaced by $E_-$ holds true, if (\ref{51}) is violated. This means that we can apply the arguments of Part I of this proof in all these cases. So for Part II we may assume (\ref{50}) and (\ref{51}).\\
{\bf Case 1:} $|B_+| \ge |A|,|C_-|$ and $|D_+| \ge |A|,|E_-|$. \\
{\bf Case 1.1:} $ |C_-| \le |\eta-\xi| \sim |\eta|$. \\
We obtain
\begin{align} \nonumber
I & \lesssim \int \frac{\tilde{F}_3(\lambda,\eta)}{ \langle \eta \rangle^{\frac{1}{2}+\epsilon}} 
\frac{\tilde{F}_0(\lambda-\tau,\eta - \xi)}{\langle C_- \rangle^{\frac{1}{2}+}}
 \langle \eta - \xi \rangle^{\frac{1}{2}+\epsilon} \frac{|\eta|^{0+}}{|\eta|^{\frac{1}{2}}} 
\\ \nonumber
& \hspace{2em} \frac{\tilde{F}_1(\sigma,\zeta)}{\langle \zeta \rangle^{\frac{1}{2} +\epsilon}}
\frac{\tilde{F}_2(\sigma - \tau, \zeta - \xi)}{\langle E_- \rangle^{\frac{1}{2}+} \langle \zeta - \xi \rangle^{\frac{1}{2}+\epsilon} |\zeta - \xi|^{\frac{1}{2}}} 
  d\sigma d\zeta d\tau d\xi d\eta d\lambda \\ \label{II.1.1}
& \lesssim \|  ( \frac{\tilde{F}_3(\lambda,\eta)}{ |\eta|^{\frac{1}{2}-}} )\check{\,\,}\|_{L^2_t L^{4-}_x} ( \frac{\tilde{F}_0(\lambda - \tau,\eta - \xi)}{\langle C_- \rangle^{\frac{1}{2}+}} )\check{\,\,}\|_{L^{\infty}_t L^2_x} \\ \nonumber
 & \hspace{2em} \|(\frac{\tilde{F}_1(\sigma,\zeta )}{\langle \zeta \rangle^{\frac{1}{2}+\epsilon}})\check{\,\,}\|_{L^2_t L^{4+}_x} \|(\frac{\tilde{F}_2(\sigma - \tau,\zeta -\xi)}{\langle E_- \rangle^{\frac{1}{2}+} \langle \zeta - \xi \rangle^{\frac{1}{2}+\epsilon} |\zeta - \xi|^{\frac{1}{2}}})\check{\,\,}\|_{L^{\infty}_t L^{\infty}_x} \\ \nonumber
 &\lesssim \prod_{i=0} ^3 \|F_i\|_{L^2_{xt}} \, .
\end{align}
{\bf Case 1.2:} $|C_-| \ge |\eta-\xi| \sim |\eta|$. \\
In this case we obtain the same bound as in Part I, Case 3.2.1 with $E_+$ replaced by $E_-$. \\
{\bf Case 2:} $|C_-| \ge |A|,|B_+|$ and $|D_+| \ge |A|,|E_-|$. \\
Using (\ref{48}) we obtain the same estimate as in case 1.2. \\
{\bf Case 3:} $|A| \ge |B_+|,|C_-|$. \\
{\bf Case 3.1:} $|C_-| \le |\xi|$. \\
We obtain
\begin{align*}
I & \lesssim \int \frac{\tilde{F}_3(\lambda,\eta)}{|\eta|^{\frac{1}{2}} \langle B_+ \rangle^{\frac{1}{2}+} \langle \eta \rangle^{\frac{1}{2}+\epsilon}} 
\frac{\tilde{F}_0(\lambda-\tau,\eta - \xi) \langle \eta - \xi \rangle^{\frac{1}{2}+\epsilon}}{\langle C_- \rangle^{\frac{1}{2}-}}
||\tau|-|\xi||^{\frac{1}{2}} 
\\
& \hspace{2em}
 \frac{\tilde{F}_1(\sigma,\zeta)}{\langle \zeta \rangle^{\frac{1}{2} +\epsilon} \langle D_+ \rangle^{\frac{1}{2}+}}
\frac{\tilde{F}_2(\sigma - \tau, \zeta - \xi)}{\langle E_- \rangle^{\frac{1}{2}+} \langle \zeta - \xi \rangle^{\frac{1}{2}+\epsilon}} 
  d\sigma d\zeta d\tau d\xi d\eta d\lambda \\
& \lesssim \int \left(\frac{\tilde{F}_3(\lambda,\eta)}{|\eta|^{\frac{1}{2}} \langle B_+ \rangle^{\frac{1}{2}+}} 
\frac{\tilde{F}_0(\lambda-\tau,\eta - \xi)}{\langle C_- \rangle^{\frac{1}{2}+}}
||\tau|-|\xi||^{\frac{1}{4}} |\xi|^{-\frac{1}{4}} \right)
\\
& \hspace{2em}\left( |\xi|^{\frac{1}{4}+}||\tau|-|\xi||^{\frac{1}{4}} \frac{\tilde{F}_1(\sigma,\zeta)}{\langle \zeta \rangle^{\frac{1}{2} +\epsilon} \langle D_+ \rangle^{\frac{1}{2}+}}
\frac{\tilde{F}_2(\sigma - \tau, \zeta - \xi)}{\langle E_- \rangle^{\frac{1}{2}+} \langle \zeta - \xi \rangle^{\frac{1}{2}+\epsilon}} \right)
  d\sigma d\zeta d\tau d\xi d\eta d\lambda \, .
\end{align*}
We take both factors in the $L^2_{xt}$-norm. We remark that in the first factor the interaction is of type $(+,+)$ because of the conjugation in its second factor $F_0$ (remark that $|C_-| = |\lambda - \tau -|\eta - \xi|| = |\tau-\lambda+|\xi-\eta||$). This means that we can apply Cor. \ref{Cor.1.2} with
$\beta_0 = -\frac{1}{4}$ , $\alpha_1 = \frac{1}{2}$ , $\alpha_2 =0$. For the second factor we apply Cor. \ref{Cor.1.1} with $\beta_0 = \frac{1}{4}+$ , $\beta_- = \frac{1}{4}$ , $\alpha_1 = \frac{1}{2}+$ , $\alpha_2 =\frac{1}{2}$. Thus we get the bound $ \prod_{i=0} ^3 \|F_i\|_{L^2_{xt}} \, .$ \\
{\bf Case 3.2:} $|C_-| \ge |\xi|$. \\
We obtain
\begin{align*}
I 
& \lesssim \|  ( \frac{\tilde{F}_3(\lambda,\eta)}{ \langle B_+ \rangle^{\frac{1}{2}+}} )\check{\,\,}\|_{L^{\infty}_t L^2_x}  \|(\frac{\tilde{F}_0(\lambda - \tau,\eta - \xi)}{|\eta - \xi|^{\frac{1}{2}}})\check{\,\,}\|_{L^2_t L^4_x} \\
 & \hspace{2em} \|(||\tau|-|\xi||^{\frac{1}{2}}|\xi|^{-\frac{1}{2}+}\frac{\tilde{F}_1(\sigma,\zeta )}{\langle \zeta \rangle^{\frac{1}{2}+\epsilon}\langle D_+ \rangle	^{\frac{1}{2}+}} \frac{\tilde{F}_2(\sigma - \tau,\zeta -\xi)}{\langle E_- \rangle^{\frac{1}{2}+} \langle \zeta - \xi \rangle^{\frac{1}{2}+\epsilon}})\check{\,\,}\|_{L^2_t L^4_x} \\
 &\lesssim \prod_{i=0} ^3 \|F_i\|_{L^2_{xt}} \, .
\end{align*}
In the last step we first used Sobolev's embedding $\dot{H}^{\frac{1}{2}}_x \subset L^4_x$ and then Cor. \ref{Cor.1.1} with $\beta_0 =0+$ , $\beta_- = \frac{1}{2}$ , $\alpha_1 = \frac{1}{2}$ , $\alpha_2 = \frac{1}{2}+$ for the last factor. \\
{\bf Case 4:} $|B_+| \ge |A|,|C_-|$ and $|A| \ge |D_+|,|E_-|$. \\
{\bf Case 4.1:} $|C_-| \le |\xi|$.\\
In this case we obtain
\begin{align*}
I 
& \lesssim \|  ( \frac{\tilde{F}_3(\lambda,\eta)}{|\eta|^{\frac{1}{2}}} )\check{\,\,}\|_{L^2_t L^4_x}  \|(\frac{\tilde{F}_0(\lambda - \tau,\eta - \xi)}{\langle C_- \rangle^{\frac{1}{2}+}})\check{\,\,}\|_{L^{\infty}_t L^2_x} \\
 & \hspace{2em} \|(||\tau|-|\xi||^{\frac{1}{2}}|\xi|^{0+}\frac{\tilde{F}_1(\sigma,\zeta )}{\langle \zeta \rangle^{\frac{1}{2}+\epsilon}\langle D_+ \rangle	^{\frac{1}{2}+}} \frac{\tilde{F}_2(\sigma - \tau,\zeta -\xi)}{\langle E_- \rangle^{\frac{1}{2}+} \langle \zeta - \xi \rangle^{\frac{1}{2}+\epsilon}|\zeta - \xi|^{\frac{1}{2}}})\check{\,\,}\|_{L^2_t L^4_x} \\
 &\lesssim \prod_{i=0} ^3 \|F_i\|_{L^2_{xt}} \, .
\end{align*}
In the last step we first used Sobolev's embedding $\dot{H}^{\frac{1}{2}}_x \subset L^4_x$ and then Cor. \ref{Cor.1.1} with $\beta_0 =\frac{1}{2}+$ , $\beta_- = \frac{1}{2}$ , $\alpha_1 = \frac{1}{2}+$ , $\alpha_2 = 1$ for the last factor. \\
{\bf Case 4.2:} $ |C_-| \ge |\xi|$.\\
We obtain
\begin{align*}
I 
& \lesssim \|  ( \frac{\tilde{F}_3(\lambda,\eta)}{\langle B_+ \rangle^{\frac{1}{2}+}} )\check{\,\,}\|_{L^{\infty}_t L^2_x}  \|F_0\|_{L^2_t L^2_x} \\
 & \hspace{2em} \|(||\tau|-|\xi||^{\frac{1}{2}}|\xi|^{-\frac{1}{2}+}\frac{\tilde{F}_1(\sigma,\zeta )}{\langle \zeta \rangle^{\frac{1}{2}+\epsilon}\langle D_+ \rangle	^{\frac{1}{2}+}} \frac{\tilde{F}_2(\sigma - \tau,\zeta -\xi)}{\langle E_- \rangle^{\frac{1}{2}+} \langle \zeta - \xi \rangle^{\frac{1}{2}+\epsilon}|\zeta - \xi|^{\frac{1}{2}}})\check{\,\,}\|_{L^2_t L^{\infty}_x} \\
 &\lesssim \prod_{i=0} ^3 \|F_i\|_{L^2_{xt}} \, .
\end{align*}
In the last step the last factor is estimated by Sobolev's embedding $\|f\|_{L^{\infty}_x} \lesssim \|(|\xi|^{1-} + |\xi|^{1+})\widehat{f}(\xi)\|_{L^2}$ and then Cor. \ref{Cor.1.1} with $\beta_0 =\frac{1}{2}+$ , $\beta_- = \frac{1}{2}$ , $\alpha_1 = \frac{1}{2}+$ , $\alpha_2 = 1$. \\
{\bf Case 5:} $|C_-| \ge |A|,|B_+|$ and $|A| \ge |D_+|,|E_-|$. \\
In this case we estimate as follows
\begin{align} \nonumber
I & \lesssim \int \frac{\tilde{F}_3(\lambda,\eta)}{ \langle B_+ \rangle^{\frac{1}{2}+} \langle \eta \rangle^{\frac{1}{2}+\epsilon}} 
\tilde{F}_0(\lambda-\tau,\eta - \xi) \langle \eta - \xi \rangle^{\frac{1}{2}+\epsilon} |\eta|^{0+}\frac{(|\eta|+|\eta - \xi|)^{\frac{1}{2}}}{|\eta|^{\frac{1}{2}} |\eta - \xi|^{\frac{1}{2}}}
\\   \nonumber
& \hspace{2em}  
 \frac{\tilde{F}_1(\sigma ,\zeta)}{\langle \zeta \rangle^{\frac{1}{2} +\epsilon} \langle D_+ \rangle^{\frac{1}{2}+}}
\frac{\tilde{F}_2(\sigma - \tau, \zeta - \xi)}{\langle E_- \rangle^{\frac{1}{2}+} \langle \zeta - \xi \rangle^{\frac{1}{2}+\epsilon} |\zeta - \xi|^{\frac{1}{2}}} ||\tau|-|\xi||^{\frac{1}{2}} 
  d\sigma d\zeta d\tau d\xi d\eta d\lambda \\ \label{II.5}
& \lesssim \|  ( \frac{\tilde{F}_3(\lambda,\eta)}{|\eta|^{\frac{1}{2}-} \langle B_+ \rangle^{\frac{1}{2}+}} )\check{\,\,}\|_{L^{\infty}_t L^{4-}_x}  \|F_0\|_{L^2_t L^2_x} \\ \nonumber
 & \hspace{2em} \|  (\frac{\tilde{F}_1(\sigma ,\zeta)(\sigma ,\zeta)(\sigma,\zeta)}{\langle \zeta \rangle^{\frac{1}{2} +\epsilon} \langle D_+ \rangle^{\frac{1}{2}+}}
\frac{\tilde{F}_2(\sigma - \tau, \zeta - \xi)}{\langle E_- \rangle^{\frac{1}{2}+} \langle \zeta - \xi \rangle^{\frac{1}{2}+\epsilon} |\zeta - \xi|^{\frac{1}{2}}} ||\tau|-|\xi||^{\frac{1}{2}})   \check{\,\,}\|_{L^2_t L^{4+}_x} \\ \nonumber
 &\lesssim \prod_{i=0} ^3 \|F_i\|_{L^2_{xt}} \, .  
\end{align}
In the last step the last factor is estimated by the embedding $\dot{H}^{\frac{1}{2}+}_x \subset L^{4+}_x$ and Cor. \ref{Cor.1.1} with $\beta_0 = \frac{1}{2}+$ , $\beta_- = \frac{1}{2}$ , $\alpha_1 = \frac{1}{2}+$ , $\alpha_2 = 1$, which completes Part II. \\
{\bf Part III.} Now we consider the case $\pm_3 = +$ , $\pm = -$ and $\pm_1=+$ , $\pm_2=+$. 

 An analogous proof works for $\pm_3 = -$ , $\pm = +$ and/or $\pm_1=-$ , $\pm_2=-$. 
 
 Again we only have to consider the case $|\xi| \ll |\eta| \sim |\eta - \xi|$. As above we also assume $|\zeta| \ge |\zeta - \xi|$ so that $|\xi| \lesssim |\zeta|$ , because the case $|\zeta| \le |\zeta - \xi|$ can be treated similarly. \\
 {\bf Case 1:} $|B_+| \ge |A|,|C_-|$ and $|D_+| \ge |A|,|E_+|$. \\
 {\bf Case 1.1:} $ |C_-| \le |\eta|$.\\
 We obtain in this case
\begin{align*}
I & \lesssim \int \frac{\tilde{F}_3(\lambda,\eta)}{ \langle \eta \rangle^{\frac{1}{2}+\epsilon} |\eta|^{\frac{1}{2}-}}
\frac{\tilde{F}_0(\lambda-\tau,\eta - \xi) \langle \eta - \xi \rangle^{\frac{1}{2}+\epsilon}}{\langle C_- \rangle^{\frac{1}{2}+}}
\\ 
& \hspace{2em}
 \frac{\tilde{F}_1(\sigma ,\zeta)}{\langle \zeta \rangle^{\frac{1}{2} +\epsilon}}
\frac{\tilde{F}_2(\sigma - \tau, \zeta - \xi)}{\langle E_+ \rangle^{\frac{1}{2}+} \langle \zeta - \xi \rangle^{\frac{1}{2}+\epsilon}} \frac{|\xi|^{\frac{1}{2}}}{|\zeta|^{\frac{1}{2}} |\zeta - \xi|^{\frac{1}{2}}} 
  d\sigma d\zeta d\tau d\xi d\eta d\lambda   
\end{align*} 
This can be bounded as in Part II, Case 1.1 by (\ref{II.1.1}) with $E_-$ replaced by $E_+$.\\
{\bf Case 1.2:} $|C_-| \ge |\eta|$. 
This case can be treated as Part I, Case 3.2.1.\\
{\bf Case 2:} $|A| \ge |B_+|,|C_-|$. 
This case is handled like Part II, Case 3. \\
{\bf Case 3:} $|C_-| \ge |A|,|B_+|$ and $|D_+| \ge |A|,|E_+|$.
We obtain the same estimate as in Case 1.2. \\
{\bf Case 4:} $ |C_-| \ge |A|,|B_+|$ and $|A| \ge |D_+|,|E_+|$.
In this case we arrive at (\ref{II.5}) as in Part II, Case 5 with $E_-$ replaced by $E_+$. \\
{\bf Case 5:} $|B_+| \ge |A|,|C_-|$ and $|A| \ge |D_+|,|E_+|$. \\
{\bf Case 5.1:} $|C_-| \le |\eta|$. \\
We obtain
\begin{align*}
I & \lesssim \|  ( \frac{\tilde{F}_3(\lambda,\eta)}{ |\eta|^{\frac{1}{2}-}} )\check{\,\,}\|_{L^2_t L^{4-}_x}
\|( \frac{F_0(\lambda - \tau,\eta - \xi)}{\langle C_- \rangle^{\frac{1}{2}+}})\check{\,\,}\|_{L^{\infty}_t L^2_x} \\
 & \hspace{2em} \|  (\frac{\tilde{F}_1(\sigma ,\zeta)}{\langle \zeta \rangle^{\frac{1}{2} +\epsilon}\langle D_+ \rangle^{\frac{1}{2}+}}\frac{\tilde{F}_2(\sigma - \tau, \zeta - \xi)}{\langle E_+ \rangle^{\frac{1}{2}+} \langle \zeta - \xi \rangle^{\frac{1}{2}+\epsilon} |\zeta - \xi|^{\frac{1}{2}}} ||\tau|-|\xi||^{\frac{1}{2}})  \check{\,\,}\|_{L^2_t L^{4+}_x} \, . 
\end{align*}
The last factor is estimated by the embedding $\dot{H}^{\frac{1}{2}+}_x \subset L^{4+}_x$ and Cor. \ref{Cor.1.1} with $\beta_0 = \frac{1}{2}+$ , $\beta_- = \frac{1}{2}$ , $\alpha_1 = \frac{1}{2}+$ , $\alpha_2 = 1$,
so that the desired estimate follows. \\
{\bf Case 5.2:} $|C_-| \ge |\eta|$. This case can be treated exactly like Case 4, so that Part III is complete. \\
{\bf Part IV.} Finally we consider the signs $\pm=+$ , $\pm_3 =+$ and $\pm_1 =+$ , $\pm_2 = -$. 

In the same way one can also treat the cases $\pm =-$ , $\pm_3 = -$ and/or $\pm_1 = -,$ $\pm_3 = +$.

We may assume as discussed above that $|\xi| \ll |\zeta| \sim |\zeta - \xi|$. \\
{\bf Case 1:} $|B_+| \ge |A|,|C_+|$ and $|D_+| \ge |A|,|E_-|$. \\
{\bf Case 1.1:} $|C_+| \le |\eta|$. \\
 We obtain in this case
\begin{align*}
I & \lesssim \int \frac{\tilde{F}_3(\lambda,\eta)}{ \langle \eta \rangle^{\frac{1}{2}+\epsilon}}
\frac{\tilde{F}_0(\lambda-\tau,\eta - \xi) \langle \eta - \xi \rangle^{\frac{1}{2}+\epsilon}}{\langle C_+ \rangle^{\frac{1}{2}+}} \frac{|\xi|^{\frac{1}{2}} |\eta|^{0+}}{|\eta|^{\frac{1}{2}} |\eta - \xi|^{\frac{1}{2}}}
\\ 
& \hspace{2em}
 \frac{\tilde{F}_1(\sigma ,\zeta)}{\langle \zeta \rangle^{\frac{1}{2} +\epsilon}}
\frac{\tilde{F}_2(\sigma - \tau, \zeta - \xi)}{\langle E_+ \rangle^{\frac{1}{2}+} \langle \zeta - \xi \rangle^{\frac{1}{2}+\epsilon}|\zeta - \xi|^{\frac{1}{2}}}
  d\sigma d\zeta d\tau d\xi d\eta d\lambda \, .
\end{align*}
Estimating $ \langle \eta - \xi \rangle \lesssim \langle \eta \rangle + \langle \xi \rangle$ we distinguish two subcases.\\
{\bf Case 1.1.1:} $|\eta| \le |\xi|$ $\Rightarrow$ $\langle \eta-\xi \rangle \lesssim \langle \xi \rangle \lesssim \langle \zeta \rangle + \langle \zeta - \xi \rangle \sim \langle \zeta \rangle$. \\
We obtain
\begin{align*}
I & \lesssim \|  ( \frac{\tilde{F}_3(\lambda,\eta)}{\langle \eta \rangle^{\frac{1}{2}+\epsilon} |\eta|^{\frac{1}{2}-}} )\check{\,\,}\|_{L^2_t L^{\infty}_x}
\|( \frac{\tilde{F}_0(\sigma,\zeta)}{\langle C_+ \rangle^{\frac{1}{2}+}})\check{\,\,}\|_{L^{\infty}_t L^2_x} \\
 & \hspace{2em} \|F_1\|_{L^2_t L^2_x} \|  ( \frac{\tilde{F}_2(\sigma - \tau, \zeta - \xi)}{\langle E_- \rangle^{\frac{1}{2}+} \langle \zeta - \xi \rangle^{\frac{1}{2}+\epsilon} |\zeta - \xi|^{\frac{1}{2}}})  \check{\,\,}\|_{L^{\infty}_t L^{\infty}_x} \, , 
\end{align*}
which gives the desired bound. \\
{\bf Case 1.1.2:} $ |\eta| \ge |\xi|$ $\Rightarrow$ $\langle \eta - \xi \rangle \lesssim \langle \eta \rangle$ , $|\xi| \le |\zeta| + |\zeta - \xi| \sim |\zeta|$. \\
We obtain
\begin{align*}
I & \lesssim \|  ( \frac{\tilde{F}_3(\lambda,\eta)}{\langle \eta \rangle^{\frac{1}{2}} |\eta|^{\frac{1}{2}-}} )\check{\,\,}\|_{L^2_t L^{\infty -}_x}
\|( \frac{F_0(\sigma,\zeta)}{\langle C_+ \rangle^{\frac{1}{2}+}})\check{\,\,}\|_{L^{\infty}_t L^2_x} \\
 & \hspace{2em} \|(\frac{\tilde{F}_1(\sigma ,\zeta)}{\langle \zeta \rangle^{\epsilon}})\check{\,\,} \|_{L^2_t L^{2+}_x} \|  ( \frac{\tilde{F}_2(\sigma - \tau, \zeta - \xi)}{\langle E_- \rangle^{\frac{1}{2}+} \langle \zeta - \xi \rangle^{\frac{1}{2}+\epsilon} |\zeta - \xi|^{\frac{1}{2}}})  \check{\,\,}\|_{L^{\infty}_t L^{\infty}_x} \, , 
\end{align*}
which leads to the desired bound. \\
{\bf Case 1.2:} $|C_+| \ge |\eta|$. \\
We obtain in this case the same bounds as in Part I, Case 3.2 with $E_+$ replaced by $E_-$.\\
{\bf Case 2:} $|A| \ge |B_+|,|C_+|$ and $|A| \ge |D_+|,|E_-|$. \\
{\bf Case 2.1:} $|C_+| \le |\eta|$. \\
We obtain in this case
\begin{align*}
I & \lesssim \int \frac{\tilde{F}_3(\lambda,\eta)}{ \langle \eta \rangle^{\frac{1}{2}+\epsilon} \langle B_+ \rangle^{\frac{1}{2}+}}
\frac{\tilde{F}_0(\lambda-\tau,\eta - \xi) \langle \eta - \xi \rangle^{\frac{1}{2}+\epsilon}}{\langle C_+ \rangle^{\frac{1}{2}+}} \frac{|\eta|^{0+} |\xi|^{\frac{1}{2}}}{|\eta|^{\frac{1}{2}} |\eta - \xi|^{\frac{1}{2}}} ||\tau|-|\xi||^{\frac{1}{2}}
\\ 
& \hspace{2em}
 \frac{\tilde{F}_1(\sigma ,\zeta)}{\langle \zeta \rangle^{\frac{1}{2} +\epsilon} \langle D_+ \rangle^{\frac{1}{2}+}}
\frac{\tilde{F}_2(\sigma - \tau, \zeta - \xi)}{\langle E_- \rangle^{\frac{1}{2}+} \langle \zeta - \xi \rangle^{\frac{1}{2}+\epsilon}|\zeta - \xi|^{\frac{1}{2}}} ||\tau|-|\xi||^{\frac{1}{2}}
  d\sigma d\zeta d\tau d\xi d\eta d\lambda \, .
\end{align*}
{\bf Case 2.1.1:} $|\xi| \ge |\eta|$. We handle this case as Part I, Case 6.1.2 with $E_+$ replaced by $E_-$.\\
{\bf Case 2.1.2:} $|\eta| \ge |\xi|$ $\Rightarrow$ $\langle \eta - \xi \rangle^{\epsilon} \lesssim \langle \eta \rangle^{\epsilon}$. \\
Similarly as before
\begin{align*}
I & \lesssim \| ( \frac{\tilde{F}_3(\lambda,\eta)}{|\eta|^{\frac{1}{2}-} \langle \eta \rangle^{\frac{1}{2}} \langle B_+ \rangle^{\frac{1}{2}+}}
\frac{\tilde{F}_0(\lambda-\tau,\eta - \xi)}{\langle C_+ \rangle^{\frac{1}{2}+}} ||\tau|-|\xi||^{\frac{1}{2}-} )\check{\,\,}\|_{L^2_t L^2_x}
\\ 
& \hspace{2em}
\| ( \frac{\tilde{F}_1(\sigma ,\zeta)}{\langle \zeta \rangle^{\epsilon} \langle D_+ \rangle^{\frac{1}{2}+}}
\frac{\tilde{F}_2(\sigma - \tau, \zeta - \xi)}{\langle E_- \rangle^{\frac{1}{2}+} \langle \zeta - \xi \rangle^{\frac{1}{2}+\epsilon}|\zeta - \xi|^{\frac{1}{2}}} ||\tau|-|\xi||^{\frac{1}{2}+} )\check{\,\,}\|_{L^2_t L^2_x} \, .
\end{align*}
Cor. \ref{Cor.1.1} with $\beta_0 = 0$ , $\beta_- = \frac{1}{2}-$ , $\alpha_1=1-$ , $\alpha_2=0$ and $\beta_0 = 0$ , $\beta_- = \frac{1}{2}+$ , $\alpha_1=0+$ , $\alpha_2=1$ for the first and second factor, respectively, gives the required estimate. \\
{\bf Case 2.2:} $|C_+| \ge |\eta|$. \\
We obtain 
\begin{align*}
I & \lesssim \int \frac{\tilde{F}_3(\lambda,\eta)}{ \langle \eta \rangle^{\frac{1}{2}+\epsilon} |\eta|^{\frac{1}{2}-} \langle B_+ \rangle^{\frac{1}{2}+}}
\tilde{F}_0(\lambda-\tau,\eta - \xi) \langle \eta - \xi \rangle^{\frac{1}{2}+\epsilon}
\\ 
& \hspace{2em}
 \frac{\tilde{F}_1(\sigma ,\zeta)}{\langle \zeta \rangle^{\frac{1}{2} +\epsilon} \langle D_+ \rangle^{\frac{1}{2}+}}
\frac{\tilde{F}_2(\sigma - \tau, \zeta - \xi)}{\langle E_- \rangle^{\frac{1}{2}+} \langle \zeta - \xi \rangle^{\frac{1}{2}+\epsilon}|\zeta - \xi|^{\frac{1}{2}}} ||\tau|-|\xi||^{\frac{1}{2}}
  d\sigma d\zeta d\tau d\xi d\eta d\lambda \, .
\end{align*}
{\bf Case 2.2.1:} $|\eta| \ge |\xi|$. This can be treated exactly as Part II, Case 5.\\
{\bf Case 2.2.2:} $|\xi| \ge |\eta|$. We handle this case as Part I, Case 5.1 with $E_+$ replaced by $E_-$.\\
{\bf Case 3:} $|B_+| \ge |A|,|C_+|$ and $|A| \ge |D_+|,|E_-|$.\\
{\bf Case 3.1:} $ |C_+| \le |\eta|$.\\
We obtain in this case
\begin{align*}
I & \lesssim \int \frac{\tilde{F}_3(\lambda,\eta)}{ \langle \eta \rangle^{\frac{1}{2}+\epsilon -}}
\frac{\tilde{F}_0(\lambda-\tau,\eta - \xi) \langle \eta - \xi \rangle^{\frac{1}{2}+\epsilon}}{\langle C_+ \rangle^{\frac{1}{2}+}} \frac{ |\xi|^{\frac{1}{2}}}{|\eta|^{\frac{1}{2}} |\eta - \xi|^{\frac{1}{2}}}
\\ 
& \hspace{2em}
 \frac{\tilde{F}_1(\sigma ,\zeta)}{|\zeta|^{\frac{1}{2}} \langle \zeta \rangle^{\frac{1}{2} +\epsilon} \langle D_+ \rangle^{\frac{1}{2}+}}
\frac{\tilde{F}_2(\sigma - \tau, \zeta - \xi)}{\langle E_- \rangle^{\frac{1}{2}+} \langle \zeta - \xi \rangle^{\frac{1}{2}+\epsilon}} ||\tau|-|\xi||^{\frac{1}{2}}
  d\sigma d\zeta d\tau d\xi d\eta d\lambda \\
& \lesssim \int \frac{\tilde{F}_3(\lambda,\eta)}{ \langle \eta \rangle^{\frac{1}{2}+\epsilon -}|\eta|^{\frac{1}{2}}}
\frac{\tilde{F}_0(\lambda-\tau,\eta - \xi) \langle \eta - \xi \rangle^{\epsilon}}{\langle C_+ \rangle^{\frac{1}{2}+}}
\\ 
& \hspace{2em}
 \frac{\tilde{F}_1(\sigma ,\zeta)}{\langle \zeta \rangle^{\frac{1}{2} +\epsilon} \langle D_+ \rangle^{\frac{1}{2}+}}
\frac{\tilde{F}_2(\sigma - \tau, \zeta - \xi)}{\langle E_- \rangle^{\frac{1}{2}+} \langle \zeta - \xi \rangle^{\frac{1}{2}+\epsilon}} ||\tau|-|\xi||^{\frac{1}{2}}
  d\sigma d\zeta d\tau d\xi d\eta d\lambda \, .
\end{align*}
{\bf Case 3.1.1:} $ |\xi| \ge |\eta|$ $\Rightarrow$ $\langle \eta - \xi \rangle^{\epsilon} \lesssim \langle \xi \rangle^{\epsilon} \lesssim \langle \zeta \rangle^{\epsilon}$. \\
We obtain
\begin{align*}
I & \lesssim \| ( \frac{\tilde{F}_3(\lambda,\eta)}{|\eta|^{\frac{1}{2}} \langle \eta \rangle^{\frac{1}{2}+\epsilon -}} )\check{\,\,}\|_{L^2_t L^{\infty}_x}
\| (\frac{\tilde{F}_0(\lambda-\tau,\eta - \xi)}{\langle C_+ \rangle^{\frac{1}{2}+}} )\check{\,\,}\|_{L^{\infty}_t L^2_x}
\\ 
& \hspace{2em}
\| ( \frac{\tilde{F}_1(\sigma ,\zeta)}{\langle \zeta \rangle^{\frac{1}{2}} \langle D_+ \rangle^{\frac{1}{2}+}}
\frac{\tilde{F}_2(\sigma - \tau, \zeta - \xi)}{\langle E_- \rangle^{\frac{1}{2}+} \langle \zeta - \xi \rangle^{\frac{1}{2}+\epsilon}} ||\tau|-|\xi||^{\frac{1}{2}} )\check{\,\,}\|_{L^2_t L^2_x} \, .
\end{align*}
The claim follow by an application of Cor. \ref{Cor.1.1} with $\beta_0 =0$ , $\beta_- = \frac{1}{2}$ , $\alpha_1 = \frac{1}{2}$ , $\alpha_2 = \frac{1}{2}$. \\
{\bf Case 3.1.2:} $ |\eta| \ge |\xi|$ $\Rightarrow$ $\langle \eta - \xi \rangle^{\epsilon} \sim |\eta-\xi|^{\epsilon} \lesssim |\eta|^{\epsilon}$. \\
We arrive at
\begin{align*}
I & \lesssim \| ( \frac{\tilde{F}_3(\lambda,\eta)}{|\eta|^{\frac{1}{2}-} \langle \eta \rangle^{\frac{1}{2}}} )\check{\,\,}\|_{L^2_t L^{\infty -}_x}
\| (\frac{\tilde{F}_0(\lambda-\tau,\eta - \xi)}{\langle C_+ \rangle^{\frac{1}{2}+}} )\check{\,\,}\|_{L^{\infty}_t L^2_x}
\\ 
& \hspace{2em}
\| ( \frac{\tilde{F}_1(\sigma ,\zeta)}{\langle \zeta \rangle^{\frac{1}{2}+\epsilon} \langle D_+ \rangle^{\frac{1}{2}+}}
\frac{\tilde{F}_2(\sigma - \tau, \zeta - \xi)}{\langle E_- \rangle^{\frac{1}{2}+} \langle \zeta - \xi \rangle^{\frac{1}{2}+\epsilon}} ||\tau|-|\xi||^{\frac{1}{2}} )\check{\,\,}\|_{L^2_t L^{2+}_x} \, .
\end{align*}
In the last factor we use the embedding $\dot{H}^{0+}_x \subset L^{2+}_x$ and then Cor. \ref{Cor.1.1} with $\beta_0 = 0+$ , $\beta_- = \frac{1}{2}$ , $\alpha_1 = \frac{1}{2}+$ , $\alpha_2 = \frac{1}{2}$.\\
{\bf Case 3.2:} $|C_+| \ge |\eta|$. \\
We obtain
\begin{align*}
I & \lesssim \int \frac{\tilde{F}_3(\lambda,\eta)}{ \langle \eta \rangle^{\frac{1}{2}+\epsilon} \langle B_+ \rangle^{\frac{1}{2}+}}
\frac{\tilde{F}_0(\lambda-\tau,\eta - \xi) \langle \eta - \xi \rangle^{\frac{1}{2}+\epsilon}}{|\eta|^{\frac{1}{2}-}}
\\ 
& \hspace{2em}
 \frac{\tilde{F}_1(\sigma ,\zeta)}{ \langle \zeta \rangle^{\frac{1}{2} +\epsilon} \langle D_+ \rangle^{\frac{1}{2}+}}
\frac{\tilde{F}_2(\sigma - \tau, \zeta - \xi)}{\langle E_- \rangle^{\frac{1}{2}+} \langle \zeta - \xi \rangle^{\frac{1}{2}+\epsilon} |\zeta - \xi|^{\frac{1}{2}}} ||\tau|-|\xi||^{\frac{1}{2}}
  d\sigma d\zeta d\tau d\xi d\eta d\lambda \, .
\end{align*}
{\bf Case 3.2.1:} $|\xi| \ge |\eta|$.\\
We obtain
\begin{align*}
I & \lesssim \| ( \frac{\tilde{F}_3(\lambda,\eta)}{|\eta|^{\frac{1}{2}-} \langle \eta \rangle^{\frac{1}{2}+\epsilon} \langle B_+ \rangle^{\frac{1}{2}+}} )\check{\,\,}\|_{L^{\infty}_t L^{\infty}_x}
\|F_0\|_{L^2_t L^2_x}
\\ 
& \hspace{2em}
\| ( \frac{\tilde{F}_1(\sigma ,\zeta)}{ \langle D_+ \rangle^{\frac{1}{2}+}}
\frac{\tilde{F}_2(\sigma - \tau, \zeta - \xi)}{\langle E_- \rangle^{\frac{1}{2}+} \langle \zeta - \xi \rangle^{\frac{1}{2}+\epsilon} |\zeta - \xi|^{\frac{1}{2}}} ||\tau|-|\xi||^{\frac{1}{2}} )\check{\,\,}\|_{L^2_t L^2_x} \, .
\end{align*}
In the last factor we use Cor. \ref{Cor.1.1} with $\beta_0 = 0$ , $\beta_- = \frac{1}{2}$ , $\alpha_1 = 0$ , $\alpha_2 = 1$.\\
{\bf Case 3.2.2:} $|\eta| \ge |\xi|$. This case is treated exactly as Part II, Case 5.\\
{\bf Case 4:} $|A| \ge |B_+|,|C_+|$ and $|D_+| \ge |A|,|E_-|$. \\
{\bf Case 4.1:} $|C_+| \le |\eta|$.\\
We obtain
\begin{align*}
I & \lesssim \int \frac{\tilde{F}_3(\lambda,\eta)}{ \langle \eta \rangle^{\frac{1}{2}+\epsilon} \langle B_+ \rangle^{\frac{1}{2}+}}
\frac{\tilde{F}_0(\lambda-\tau,\eta - \xi) \langle \eta - \xi \rangle^{\frac{1}{2}+\epsilon}|\xi|^{\frac{1}{2}}}{\langle C_+ \rangle^{\frac{1}{2}+} |\eta - \xi|^{\frac{1}{2}}|\eta|^{\frac{1}{2}-}} ||\tau|-|\xi||^{\frac{1}{2}}
\\ 
& \hspace{2em}
 \frac{\tilde{F}_1(\sigma ,\zeta)}{ \langle \zeta \rangle^{\frac{1}{2} +\epsilon}}
\frac{\tilde{F}_2(\sigma - \tau, \zeta - \xi)}{\langle E_- \rangle^{\frac{1}{2}+} \langle \zeta - \xi \rangle^{\frac{1}{2}+\epsilon} |\zeta - \xi|^{\frac{1}{2}}}
  d\sigma d\zeta d\tau d\xi d\eta d\lambda \, .
\end{align*}
{\bf Case 4.1.1:} $|\xi| \ge |\eta|$ $\Rightarrow$ $\langle \eta - \xi \rangle \lesssim \langle \xi \rangle \lesssim \langle \zeta \rangle$ , $|\xi| \lesssim |\zeta|$. We obtain the same estimate as in Part I, Case 3.1.2 with $E_+$ replaced by $E_-$.\\
{\bf Case 4.1.2:} $ |\eta| \ge |\xi|$. \\
We recall our tacid assumption $|\zeta| \ge |\zeta -\xi|$ so that $|\xi| \lesssim |\zeta|$, and thus obtain
\begin{align*}
I & \lesssim \| ( \frac{\tilde{F}_3(\lambda,\eta)}{|\eta|^{\frac{1}{2}-} \langle \eta \rangle^{\frac{1}{2}} \langle B_+ \rangle^{\frac{1}{2}+}} \frac{\tilde{F}_0(\lambda - \tau,\eta - \xi)}{\langle C_+ \rangle^{\frac{1}{2}+}} ||\tau|-|\xi||^{\frac{1}{2}} |\xi|^{-\epsilon})\check{\,\,}\|_{L^2_t L^2_x}
\\ 
& \hspace{2em}
\|F_1\|_{L^2_{xt}}
\| (\frac{\tilde{F}_2(\sigma - \tau, \zeta - \xi)}{\langle E_- \rangle^{\frac{1}{2}+} \langle \zeta - \xi \rangle^{\frac{1}{2}+\epsilon} |\zeta - \xi|^{\frac{1}{2}}} )\check{\,\,}\|_{L^{\infty}_t L^{\infty}_x} \, .
\end{align*}
An application of Cor. \ref{Cor.1.1} with $\beta_0 = -\epsilon$ , $\beta_- = \frac{1}{2}$ , $\alpha_1 = 1-\epsilon$ , $\alpha_2 = 0$ gives the desired bound. \\
{\bf Case 4.2:} $ |C_+| \ge |\eta|$. \\
We obtain the same bounds as in Part I, Case 3.2 with $E_+$ replaced by $E_-$. \\
{\bf Case 5:}  $|C_+| \ge |A|,|B_+|$ and $|D_+| \ge |A|,|E_-|$. \\
We obtain
\begin{align*}
I & \lesssim \int \frac{\tilde{F}_3(\lambda,\eta)}{ \langle \eta \rangle^{\frac{1}{2}+\epsilon} \langle B_+ \rangle^{\frac{1}{2}+}}
\tilde{F}_0(\lambda-\tau,\eta - \xi) \langle \eta - \xi \rangle^{\frac{1}{2}+\epsilon} \frac{|\xi|^{\frac{1}{2}} |\eta|^{0+}}{|\eta|^{\frac{1}{2}} |\eta - \xi|^{\frac{1}{2}}}
\\ 
& \hspace{2em}
 \frac{\tilde{F}_1(\sigma ,\zeta)}{ \langle \zeta \rangle^{\frac{1}{2} +\epsilon}}
\frac{\tilde{F}_2(\sigma - \tau, \zeta - \xi)}{\langle E_- \rangle^{\frac{1}{2}+} \langle \zeta - \xi \rangle^{\frac{1}{2}+\epsilon} |\zeta - \xi|^{\frac{1}{2}}}
  d\sigma d\zeta d\tau d\xi d\eta d\lambda \, .
\end{align*}
{\bf Case 5.1:} $|\xi| \ge |\eta|$ $\Rightarrow$ $\langle \eta - \xi \rangle^{\epsilon} |\xi|^{\frac{1}{2}} \lesssim \langle \xi \rangle^{\frac{1}{2}+\epsilon} \lesssim \langle \zeta \rangle^{\frac{1}{2}+\epsilon}$. \\
This implies the same bound as in Part I, Case 4.2. \\
{\bf Case 5.2:} $|\eta| \ge |\xi|$.\\
We obtain the estimate
\begin{align*}
I & \lesssim \|  ( \frac{\tilde{F}_3(\lambda,\eta)}{\langle \eta \rangle^{\frac{1}{2}} |\eta|^{\frac{1}{2}-} \langle B_+ \rangle^{\frac{1}{2}+}} )\check{\,\,}\|_{L^{\infty}_t L^{\infty -}_x}
\|F_0\|_{L^2_t L^2_x} \\
 & \hspace{2em} \|( \frac{\tilde{F}_1(\sigma ,\zeta)}{\langle \zeta \rangle^{\epsilon}} )\check{\,\,} \|_{L^2_t L^{2+}_x} \|  ( \frac{\tilde{F}_2(\sigma - \tau, \zeta - \xi)}{\langle E_- \rangle^{\frac{1}{2}+} \langle \zeta - \xi \rangle^{\frac{1}{2}+\epsilon} |\zeta - \xi|^{\frac{1}{2}}})  \check{\,\,}\|_{L^{\infty}_t L^{\infty}_x} \, , 
\end{align*}
which implies the desired bound. \\
{\bf Case 6:} $|C_+| \ge |A|,|B_+|$ and $|A| \ge |D_+|,|E_-|$. \\
In this case we obtain
\begin{align*}
I & \lesssim \int \frac{\tilde{F}_3(\lambda,\eta)}{ \langle \eta \rangle^{\frac{1}{2}+\epsilon} \langle B_+ \rangle^{\frac{1}{2}+}}
\tilde{F}_0(\lambda-\tau,\eta - \xi) \langle \eta - \xi \rangle^{\frac{1}{2}+\epsilon} \frac{|\xi|^{\frac{1}{2}} |\eta|^{0+}}{|\eta|^{\frac{1}{2}} |\eta - \xi|^{\frac{1}{2}}}
\\ 
& \hspace{2em}
 \frac{\tilde{F}_1(\sigma ,\zeta)}{ \langle \zeta \rangle^{\frac{1}{2} +\epsilon} \langle D_+ \rangle^{\frac{1}{2}+}}
\frac{\tilde{F}_2(\sigma - \tau, \zeta - \xi)}{\langle E_- \rangle^{\frac{1}{2}+} \langle \zeta - \xi \rangle^{\frac{1}{2}+\epsilon} |\zeta - \xi|^{\frac{1}{2}}} ||\tau|-|\xi||^{\frac{1}{2}}
  d\sigma d\zeta d\tau d\xi d\eta d\lambda \, .
\end{align*}
{\bf Case 6.1:} $ |\xi| \ge |\eta|$. \\
We obtain the estimate
\begin{align*}
I & \lesssim \|  ( \frac{\tilde{F}_3(\lambda,\eta)}{\langle \eta  \rangle^{\frac{1}{2}+\epsilon} |\eta|^{\frac{1}{2}-} \langle B_+ \rangle^{\frac{1}{2}+}} )\check{\,\,}\|_{L^{\infty}_t L^{\infty}_x}
\|F_0\|_{L^2_t L^2_x} \\
 & \hspace{2em} \|( \frac{\tilde{F}_1(\sigma ,\zeta)}{\langle D_+ \rangle^{\frac{1}{2}+}} \frac{\tilde{F}_2(\sigma - \tau, \zeta - \xi)}{\langle E_- \rangle^{\frac{1}{2}+} \langle \zeta - \xi \rangle^{\frac{1}{2}+\epsilon} |\zeta - \xi|^{\frac{1}{2}}} ||\tau|-|\xi||^{\frac{1}{2}})  \check{\,\,}\|_{L^2_t L^2_x} \, , 
\end{align*}
which is further estimated by use of Cor. \ref{Cor.1.1} with $\beta_0 = 0$ , $\beta_- = \frac{1}{2}$ , $\alpha_1 = 0$ , $\alpha_2 = 1$. \\
{\bf Case 6.2:} $|\eta| \ge |\xi|$.
We obtain in this case the same estimate as in Part I, Case 5.2 with $E_+$ replaced by $E_-$, so that the proof is now complete.\\
\end{proof}
\begin{proof}[{\bf Proof of Theorem \ref{Theorem1.2}}]
Let $\psi_{\pm} \in C^0([0,T],H^s({\mathbb R}^2))$ be given, where $s= \frac{3}{4}+\epsilon$, $\epsilon > 0$ arbitrarily small, $T \le 1$. Then $\psi_{\pm} \in X_{\pm}^{s,0}[0,T] = L^2([0,T],H^s)$. By Prop. \ref{Prop.1.4} we obtain (with $\psi = \psi_+ + \psi_-$):
\begin{equation}
\label{60}
 \|\psi_{\pm}\|_{X_{\pm}^{\frac{1}{4}+2\epsilon,1}[0,T]} \lesssim\|\psi_{\pm}(0)\|_{H^s} + \|\Pi_{\pm} (\langle \beta \psi,\psi \rangle \psi)\|_{X_{\pm}^{\frac{1}{4}+2\epsilon,0}[0,T]} \, .
 \end{equation}
 By the generalized H\"older inequality
 $$ \|\langle \beta \psi,\psi \rangle \psi\|_{H^{\frac{1}{4}+2\epsilon}_x} \lesssim \|\psi\|^2_{L^{2p}_x} \| \psi \|_{H_x^{\frac{1}{4}+2\epsilon , q}} $$
 with $ \frac{1}{p} =\frac{1}{4}-\frac{\epsilon}{2} $ and $\frac{1}{q} = \frac{1}{4} + \frac{\epsilon}{2}$. Sobolev's embedding gives $H^{\frac{3}{4}+\epsilon}_x \subset L^{2p}_x$, because $\frac{1}{2p} \ge \frac{1}{2}-\frac{\frac{3}{4}+\epsilon}{2} = \frac{1}{8} - \frac{\epsilon}{2}$, and $H^{\frac{3}{4}+\epsilon}_x \subset H^{\frac{1}{4}+2\epsilon,q}_x$. Consequently
 \begin{align*}
\|\Pi_{\pm} (\langle \beta \psi,\psi \rangle \psi\|_{X_{\pm}^{\frac{1}{4}+2\epsilon,0}[0,T]} =  \|\Pi_{\pm} (\langle \beta \psi,\psi \rangle \psi\|_{L^2([0,T],H_x^{\frac{1}{4}+2\epsilon})} &\lesssim \|\psi\|_{L^{\infty}([0,T],H_x^{\frac{3}{4}+\epsilon})}^3 \\
&< \infty \, . 
\end{align*}
By (\ref{60}) this implies $\psi_{\pm} \in X_{\pm}^{\frac{1}{4}+2\epsilon,1}[0,T]$. Interpolation with $\psi_{\pm} \in X_{\pm}^{\frac{3}{4}+\epsilon,0}[0,T]$ gives (for interpolation parameter $\Theta = \frac{1}{2}+$: $\psi_{\pm} \in X_{\pm}^{\frac{1}{2} + \frac{3}{2}\epsilon-,\frac{1}{2}+}[0,T]$. In this class, however, uniqueness holds by Theorem \ref{Theorem1.1}, which shows that our solution is (unconditionally) unique in $C^0([0,T],H^s)$ for any $s > \frac{3}{4}$.
\end{proof}


\begin{thebibliography}{999999}
\bibitem[AFS]{AFS} P. d'Ancona, D. Foschi, and S. Selberg: {\sl Null structure 
and 
almost optimal local regularity for the Dirac-Klein-Gordon system}. 
Journal of the EMS 9 (2007), 877-898
\bibitem[AFS1]{AFS1} P. d'Ancona, D. Foschi, and S. Selberg: {\sl Local 
well-posedness below the charge norm for the Dirac-Klein-Gordon system in two 
space dimensions}. Journal Hyperbolic Diff. Equations 4 (2007), 295-330
\bibitem[C]{C} T. Candy: {\sl Global existence for an $L^2$ critical nonlinear Dirac equation in one dimension}.  Adv. Differential Equations 16 (2011),  643-666 
\bibitem[D]{D} V. Delgado: {\sl Global solutions of the Cauchy problem for the (classical) coupled Maxwell-Dirac and other nonlinear Dirac equations in one space dimension}. Proc. AMS 69 (1978), 289-296 
\bibitem[EV]{EV} M. Escobedo and L. Vega: {\sl A semilinear Dirac equation in $H^s({\mathbb R}^3)$ for $s>1$}. Siam J. Math. Anal. 28 (1997), 338-362
\bibitem[FFK]{FFK} R. Finkelstein, C. Fronsdal and P. Kaus: {\sl Nonlinear spinor field}. Phys. Rev. 103 (1956), 1571-1579
\bibitem[FLR]{FLR} R. Finkelstein, R. LeLevier and M. Ruderman: {\sl Nonlinear spinor fields}. Phys. Rev. 83 (1951), 326-333
\bibitem[FK]{FK} D. Foschi and S. Klainerman: {\sl Homogeneous $L^2$ bilinear 
estimates for wave equations}. Ann. Scient. ENS $4^{e}$ serie 23 (2000), 
211-274
\bibitem[GV]{GV} J. Ginibre and G. Velo: {\sl Generalized Strichartz 
inequalities for the wave equation}. J. Funct. Analysis 133 (1995), 50-68
\bibitem[GN]{GN} D. Gross and A. Neveu: {\sl Dynamical symmetry breaking in asymptotically free field theories}. Phys. Rev. D 10 (1974), 3235-3253
\bibitem[GP]{GP} A. Gr\"unrock and H. Pecher: {\sl Global solutions for the Dirac-Klein-Gordon system in two space dimensions}. Comm. Partial Differential Equations 35 (2010), 89–-112
\bibitem[MNNO]{MNNO} S. Machihara, M. Nakamura, K. Nakanishi and T. Ozawa: {\sl Endpoint Strichartz estimates and global solutions for the nonlinear Dirac equation}. J. Funct. Anal. 219 (2005), 1-20
\bibitem[MNO]{MNO} S. Machihara, K. Nakanishi and T. Ozawa: {\sl Small global solutions and the relativistic limit for the nonlinear Dirac equation}. Rev. Math. Iberoamericana 19 (2003), 179-194 
\bibitem[S]{S} S. Selberg: {\sl Multilinear spacetime estimates and 
applications 
to local existence theory for nonlinear wave equations}. Ph.D. thesis, 
Princeton 
Univ. 1999
\bibitem[ST]{ST} S. Selberg and A. Tesfahun: {\sl Low regularity well-posedness for some nonlinear Dirac equations in one space dimension}. Diff. Int. Equ. 23 (2010), 265-278
\bibitem[So]{So} M. Soler: {\sl Classical, stable, nonlinear spinor field with positive rest energy}. Phys. Rev. D 1 (1970), 2766-2769
\bibitem[T]{T} W.E. Thirring: {\sl A soluble relativistic field theory}. Ann. Physics 3 (1958), 91-112
\end{thebibliography}
\end{document}